\documentclass[smallcondensed]{svjour3}
\smartqed
\usepackage{amsmath,amsfonts,mathrsfs}
\usepackage{graphicx}
\usepackage{geometry}
\usepackage{graphicx}
\usepackage{amssymb}
\usepackage{epstopdf}
\usepackage{pgf,tikz}
\usepackage{latexsym}

\newcommand{\ignore}[1]{}

\newcommand{\pmax}{p_{\text{\rm max}}}
\newcommand{\pmin}{p_{\text{\rm min}}}
\newcommand{\phan}{p_{\text{\rm han}}}

\newcommand{\oR}{{\mathbb R}}
\newcommand{\oN}{{\mathbb N}}

\newcommand{\PP}{{\mathcal P}}
\newcommand{\II}{{\mathcal I}}
\newcommand{\ux}{\underline{x}}

\newcommand{\CC}{{\mathcal C}}
\newcommand{\HH}{{\mathcal H}}
\newcommand{\DEF}{\text{\rm defect}}
\newcommand{\ST}{\text{\rm ST}}
\newcommand{\FR}{\text{\rm FR}}
\newcommand{\rk}{{\rm rk_H}}

\newcommand{\sa}{\text{\rm sa}}
\newcommand{\dd}{\text{\rm deg}}
\newcommand{\SSS}{\mathcal{S}}

\newcommand{\tV}{{\widehat V}}
\newcommand{\C}{{\mathcal{C}}}
\newcommand{\M}{{\mathcal M}}

\newcommand{\ls}{\text{ls}}
\newcommand{\las}{\text{las}}

\journalname{}
\begin{document}

\title{Handelman's hierarchy for the maximum stable set problem}

\author{Monique Laurent         \and
        Zhao Sun}
\institute{Centrum Wiskunde \& Informatica (CWI), Amsterdam and Tilburg University \at
              CWI, Postbus 94079, 1090 GB Amsterdam, The Netherlands \\
              Tel.: +31-20-5924105\\
              \email{M.Laurent@cwi.nl}
           \and
           Tilburg University \at
              PO Box 90153, 5000 LE Tilburg, The Netherlands \\
              Tel.: +31-13-4663313\\
              Fax:  +31-13-4663280\\
              \email{z.sun@uvt.nl}
}

\date{Received: date / Accepted: date}

\maketitle

\begin{abstract}
The maximum stable set problem is a well-known NP-hard problem
in combinatorial optimization, which can be formulated as the
maximization of a quadratic square-free polynomial over the
(Boolean) hypercube. We investigate a hierarchy of linear programming
relaxations for this problem, based on a result of Handelman
showing that a positive polynomial over a polytope with non-empty interior can be
represented as conic combination of products of the linear
constraints defining the polytope. We relate the
rank of Handelman's hierarchy with structural properties of
graphs.
In particular we show a relation to fractional clique covers which we use to upper bound  the Handelman rank for perfect graphs and determine its exact value in the vertex-transitive case.
Moreover we show two upper bounds on the Handelman rank in terms of the (fractional) stability number of the graph and compute the Handelman rank for several classes of graphs including odd cycles and wheels and their complements.
We also point out  links to several other  linear and semidefinite programming hierarchies.


\end{abstract}

\section{Introduction}

In this paper we consider the maximum stable set problem,  a well-known NP-hard problem in combinatorial optimization.
We study a global optimization approach, based on reformulating the maximum stability number $\alpha(G)$ of a  graph $G$  as the maximum of a (square-free) quadratic polynomial on the hypercube $[0,1]^n$, as in relation (\ref{popalpha}) below.
 We investigate a hierarchy of linear programming bounds, motivated by a result of Handelman \cite{DH88} for certifying positive polynomials on the hypercube.  While several other linear or semidefinite programming hierarchical relaxations exist, a main motivation for focusing on the relaxations of Handelman type is that they appear to be easier to analyze. Indeed, explicit error bounds have been given for general polynomials in \cite{KL10} and sharper bounds that apply at any order of relaxation have been given in \cite{PH11,PH12} for square-free quadratic polynomials, as we will recall below.
Moreover, we focus on the maximum stable set problem, since it  is fundamental in the sense
that  any polynomial optimization problem on the hypercube can
be transformed into a maximum stable set problem using the
so-called conflict graph \cite{BB89}. Moreover, Cornaz and Jost \cite{CJ08} give a direct explicit reformulation for the graph coloring problem as an instance of maximum  stable set problem.

Algebraic approaches for the maximum stable set problem have been long studied; see e.g. the early work of Lov\'asz \cite{Lov94} and the more recent work of De Loera et al. \cite{LLMO09}, where    Hilbert's Nullstellensatz plays a central role to show the non-existence of a solution to a system of polynomial equations.
  For instance, \cite{LLMO09} uses the polynomial system: $x_i-x_i^2=0$ for $i\in V(G)$,
$x_ix_j=0$ for $ij\in E(G)$ and $\sum_{i\in V(G)}x_i=k$, to encode the question of existence of a stable set of size $k$ in $G$. For $k\ge \alpha(G)+1$ this sytem is infeasible and
\cite{LLMO09}  gives an explicit Nullstellensatz certificate certifying this  and  such certificates can be searched using Gaussian elimination (or linear programming).
Other algebraic approaches, based on finding conditions for expressing positivity of polynomials, permit to construct upper bounds for the stability number.
Depending on the type of positivity certificates one finds linear or semidefinite programming bounds (cf. e.g. \cite{GPT,KP02,Las02,Lau03,PVZ12,SA90}).
In this paper we focus on the Handelman approach, where one searches for positivity certificates obtained as conic combinations of the linear polynomials defining the hypercube.
This approach for the maximum stable set problem  was
initiated by Park and Hong \cite{PH12} (also in \cite{PH11} for
the maximum cut problem) and we will extend several of their
results.

We now introduce the Handelman hierarchy for polynomial optimization problems and recall some known results for optimization on the standard simplex and on the hypercube.

\subsection{Polynomial optimization}

Given polynomials $p,g_1,\ldots,g_m\in \oR[x]$ in $n$
variables $x=(x_1,\ldots,x_n)$, we consider the following {\em polynomial
optimization problem}:
\begin{equation}\label{problem1}
\pmax= \max \ p(x) \ \text{ s.t. } \ x\in K=\{x\in \oR^n: g_1(x)\ge 0,\ldots,g_m(x)\ge 0\},
\end{equation}
which asks to maximize  $p$ over the basic closed semialgebraic set
$K$.
This is an NP-hard problem, since it contains e.g.
 the maximum stable set problem and the maximum cut problem, two well-known NP-hard problems.  Both problems can indeed be formulated as instances of (\ref{problem1}) where $p$ is a
 quadratic polynomial and $K=[0,1]^n$ is the hypercube. Namely, given a graph $G=(V,E)$, the maximum cardinality $\alpha(G)$ of a stable set in $G$ can be computed via the polynomial optimization problem:
 \begin{equation}\label{popalpha}
 \alpha(G)=\max_{x\in [0,1]^n} \sum_{i\in V}x_i-\sum_{ij \in E} x_ix_j,
 \end{equation}
 and the maximum cardinality of a cut in $G$ can be computed via the following problem:
 \begin{equation}\label{popmaxcut}
 { \rm mc}(G)=\max_{x\in [0,1]^n}  \sum_{i\in V}\deg(i) x_i-2\sum_{ij\in E} x_ix_j,
 \end{equation}
 where $\deg(i)$ denotes the degree of node $i$ in $G$. See e.g. \cite{PH11,PH12} and Proposition 2 below.

With $\mathscr{P}(K)$ denoting the set of real polynomials that
are nonnegative on the set $K$, problem (\ref{problem1}) can be rewritten
as
\begin{equation*}\label{problem2}
\pmax=\min\ \lambda\ \ \text{\rm{s.t.}}\ \
\lambda-p\in \mathscr{P}(K).
\end{equation*}
A popular approach in the recent years is based on replacing
the (hard to test) positivity condition $\lambda-p\in \mathscr{P}(K)$
by a tractable, sufficient condition for positivity. For
instance,  one may search for positivity certificates of the
form $\lambda -p =\sum_{\alpha \in \oN^m} c_\alpha
g_1^{\alpha_1}\cdots g_m^{\alpha_m}$, where the multipliers
$c_\alpha$ are nonnegative scalars, which leads to the so-called
Handelman hierarchy of linear programming relaxations for
(\ref{problem1}).
When the $g_j$'s are linear  polynomials and $K$ is a polytope, the asymptotic convergence to $\pmax$
is guaranteed by the following  result of Handelman \cite{DH88}.

\begin{theorem}\cite{DH88}\label{theoHan}
Assume that $g_1,\dots,g_m\in\oR[x]$ are linear polynomials and that the set
\begin{equation}\label{eqK}
K=\{x\in\oR^n:g_1(x)\ge0,\dots,g_m(x)\ge0\}
\end{equation}
is compact and has a non-empty interior. Then for any polynomial $p\in\oR[x]$ strictly positive on $K$, $p$ can be written as
$\sum_{\alpha\in\oN^m} c_\alpha g_1^{\alpha_1}\cdots g_m^{\alpha_m}$
for some nonnegative scalars $ c_{\alpha}$.
\end{theorem}
In the case of the hypercube $K=\{x\in\oR^n: 0\le x_i\le 1 \ \forall i\in [n]\}$, this result was shown already earlier by Krivine \cite{Kr64}.

Alternatively, one
may search for positivity certificates of the form
$\sum_{\alpha\in \oN^m} s_\alpha g_1^{\alpha_1}\cdots
g_m^{\alpha_m}$ (or of the simpler form
$s_0+\sum_{j=1}^ms_jg_j$), where the multipliers $s_\alpha$ (or
$s_0,s_j$) are now sums of squares of polynomials. This leads
to the Lasserre hierarchy 
of semidefinite programming
relaxations for (\ref{problem1}), whose asymptotic convergence
is guaranteed for $K$ compact  (satisfying an additional Archimedean condition) by  results of  real algebraic
geometry (see e.g.  \cite{Las02,Lau09}).

Although the Lasserre hierarchy is stronger, it is more
difficult to analyze and computationally more expensive as it
relies on semidefinite programming. This motivates the study of
the linear programming based Handelman hierarchy which is generally easier to
analyze, and might yet provide some insightful information,
also for the SDP based hierarchies which dominate it. Some
results have been proved on the convergence rate in the case
when $K$ is the standard simplex or the hypercube $[0,1]^n$,
which we  recall below.

\subsection{The Handelman hierarchy}\label{sechan}

We now present a hierarchy of linear relaxations for problem  (\ref{problem1}), which is motivated by the  above mentioned result of Handelman  for certifying positivity of polynomials on a semialgebraic set $K$ of the form (\ref{eqK}).
We let $g$ denote the set of polynomials $g_1,\ldots,g_m$.
For an integer
$t\ge 1$, define the {\em Handelman set of order t} as
\begin{equation*}
\HH_t(g):=\left\{\sum_{\alpha\in \oN^m: |\alpha|\le t} c_{\alpha} g^\alpha: c_\alpha \ge 0\right\}
\end{equation*}
 and the
corresponding {\em Handelman bound of order $t$} as
\begin{equation*}\label{eqphan}
\phan^{(t)}:=\inf \{\lambda: \lambda -p \in \HH_t(g)\}.
\end{equation*}
Clearly, any polynomial in $\HH_t(g)$ is nonnegative on $K$ and one has the following chain of inclusions:
 $$\mathcal{H}_1(g)\subseteq \ldots \subseteq \mathcal{H}_t(g)\subseteq\HH_{t+1}(g) \subseteq \ldots \subseteq \mathscr{P}(K),$$
giving the chain of inequalities:
 $\pmax\le \phan^{(t+1)}\le  \phan^{(t)}\le\dots\le \phan^{(1)}$ for $t\ge 1$. When $K$ is a polytope with non-empty interior and $g_1,\ldots,g_m$ are linear polynomials, the asymptotic convergence of the bounds $\phan^{(t)}$ to $\pmax$ as the order $t$ increases is guaranteed by  Theorem \ref{theoHan} above.
We mention  two cases where results are known about the quality of the Handelman bounds, when $K$ is the standard simplex or the hypercube.

\subsubsection*{Application to optimization on  the simplex.}
We first consider the case when $K=\Delta$ is the standard simplex $\Delta=\{x\in \oR^n: x\ge 0, \ \sum_{i=1}^n x_i=1\}.$
Define the polynomial  $\sigma = \sum_{i=1}^n x_i$. Let $\langle 1-\sigma\rangle$ denote the ideal in $\oR[x]$ generated by the polynomial $1-\sigma$
and, for an integer $t$, let $\langle 1-\sigma\rangle_t$ denote its truncation at degree $t$, consisting of all polynomials of the form $u(1-\sigma)$ where $u\in \oR[x]$ has degree at most $t-1$. Moreover, let $\oR_+[x]$ denote the set of polynomials with nonnegative coefficients and $\oR_+[x]_t$ its subset consisting of polynomials of degree at most $t$.
 With $g$ standing for the set of polynomials $x_1,\ldots, x_n, \pm (1-\sigma)$,  one can easily see that the Handelman set of order $t$ is given by
$$\HH_t(g)= \oR_+[x]_t +\langle 1-\sigma\rangle_t.$$ Suppose we wish to maximize $p$ over $\Delta$, where $p\in \oR[x]$ is a polynomial of degree $d$ which we can assume to be homogeneous without loss of generality.
It turns out that the corresponding Handelman bound $\phan^{(t)}$ coincides with the LP bound studied in \cite{Fay03,KLP06}, based on P\'olya's positivity certificate and defined as follows:
$$\inf \{\lambda: (\lambda \sigma^d-p)\sigma^{t-d} \in \oR_+[x]\}.$$ This follows from the following lemma (based on similar arguments as in
\cite{KLP05}).

\begin{lemma}\label{lemsimplex}
Let $p$ be a homogeneous polynomial of degree $d$,  $\lambda\in
\oR$  and an integer $t\ge d$. Then, $\lambda-p\in \oR_+[x]_t
+\langle 1-\sigma\rangle_t$ if and only if $(\lambda
\sigma^d-p)\sigma^{t-d} \in \oR_+[x].$ Therefore, $\phan^{(t)} = \inf \{\lambda: (\lambda \sigma^d-p)\sigma^{t-d} \in \oR_+[x]\}.$
\end{lemma}

\begin{proof}
Assume $(\lambda \sigma^d-p)\sigma^{t-d} \in \oR_+[x].$ By
writing $\sigma = 1 + ( \sigma -1)$ and expanding the products
$\sigma ^d$ and $\sigma ^t$, one obtains a decomposition of
$\lambda-p $ in $\oR_+[x]_t +\langle 1-\sigma\rangle_t$.
Conversely, assume that $\lambda-p\in \oR_+[x]_t +\langle
1-\sigma\rangle_t$. This implies that $\lambda\sigma^d-p  = f
+u(1-\sigma)$, where $f\in \oR_+[x]_t $ and $u\in
\oR[x]_{t-1}$. By evaluating both sides at $x/\sigma$ and
multiplying throughout by $\sigma^t$, we obtain that
$\sigma^{t-d}(\lambda \sigma^d-p) = f(x/\sigma) \sigma^t \in
\oR_+[x]$, since $f$ has degree at most $t$.
\qed\end{proof}

Therefore the results of de Klerk, Laurent and Parrilo \cite{KLP06} apply and give the following error estimates for the Handelman bound of order $t\ge d$:
\begin{eqnarray*}
\phan^{(t)} -\pmax
\le  d^d {2d-1\choose d} \frac{{d\choose 2}}{t-{d\choose 2}} (\pmax-\pmin),
\end{eqnarray*}
where $\pmin$ is the minimum value of $p$ over the simplex $\Delta$.

\subsubsection*{Application to optimization on the hypercube.}
We now turn to the case when $K=[0,1]^n$ is the hypercube.
Using Bernstein approximations, de Klerk and Laurent \cite{KL10} have shown the following error estimates for the Handelman hierarchy.
If $p$ is a polynomial of degree $d$ and $r\ge 1$ is an integer then the Handelman bound of order $t=rn$ satisfies:
$$\phan^{(rn)}-\pmax \le {L(p) \over r} {d+1\choose 3} n^d,$$
setting $L(p)= \max_{\alpha}  {\alpha!\over |\alpha|!}|p_\alpha|$.
In the quadratic case a better estimate can be shown.

\begin{theorem}\cite[Proposition 3.2]{KL10}
Let $p=x^TAx+b^Tx$ be a quadratic polynomial. For any integer $r\ge 1$,
$$\phan^{(rn)}-\pmax \le \frac{-\sum_{i:A_{ii}<0}A_{ii}}{r}.$$
\end{theorem}

We observe that the above results hold only for relaxations of order $t\ge n$. Moreover, if $p$ is a square-free quadratic polynomial (i.e., $A_{ii}=0$ for all $i$), then  equality
$\pmax=\phan^{(n)}$ holds and the Handelman relaxation of order $n$ gives the exact value $\pmax$.  This is consistent with the fact that a square-free polynomial  takes the same maximum value on the  hypercube $[0,1]^n$ as on the Boolean hypercube $\{0,1\}^n$.

Using a combinatorial version of  Bernstein approximations, Park and Hong \cite{PH12} can analyze  the Handelman bound of any order $t\le n$,  in the quadratic square-free case. They show the following result (see Section~\ref{secerror} for a proof).

\begin{theorem}\cite{PH12}\label{theoPH}
 Let $p=x^TAx+b^Tx$ be a quadratic polynomial which is square-free, i.e., $A_{ii}=0$ for all $i\in [n]$. Assume moreover that $A_{ij}\le 0$ for all $i\ne j\in [n]$.
Then, for any integer $2\le t\le n$,
$$\phan^{(t)} \le {n\over t} \pmax.$$
\end{theorem}

\subsection{Contribution of the paper}
The error analysis from Theorem \ref{theoPH} applies in particular to the bounds obtained by applying the Handelman hierarchy to the formulation (\ref{popalpha}) of the maximum stable set problem  and to the formulation (\ref{popmaxcut}) of the maximum cut problem \cite{PH11,PH12}, whereas no error analysis is known for other (potentially stronger) linear or semidefinite programming hierarchies. This is one of the main motivations for investigating the Handelman hierarchy. Park and Hong \cite{PH11,PH12} give some preliminary results on the {\em rank} of the Handelman hierarchy, defined as  the smallest order $t$ for which the Handelman bound is exact.  In particular, they show that when applied to  both the maximum stable set and cut problems, the Handelman hierarchy has rank 2 for  bipartite graphs
and rank 3 for   odd cycles (in the unweighted case) and they ask whether these results extend to weighted graphs.
We  give an affirmative answer to this open question.

The paper is devoted to the Handelman hierarchy applied to the formulation (\ref{popalpha}) of the maximum stable set problem.
In particular, we  bound the rank of the Handelman hierarchy  for several  graph classes, including perfect graphs, odd cycles and wheels, and their complements, in  the general weighted case. Moreover we  show that the Handelman bound of order 2 is equal to the fractional stability number  (see Theorem \ref{thmrho2g}). We  also prove two different upper bounds for the Handelman rank for a weighted graph, one in terms of the (unweighted) stability number and one in terms of the weighted stability and fractional stability numbers (see Theorem \ref{thmupperbound} and Corollary \ref{2nd upper bound}).
 For this we develop the following two main tools.

First we show a relationship between the Handelman bound of order $t$ and the fractional $t$-clique cover number, at any given order $t\ge 2$, by constructing explicit decompositions in the Handelman set of order $t$ from clique covers.
At the smallest order $t=2$, we show that both bounds coincide, which implies that the Handelman bound of order 2 coincides with the fractional stable set number. Additionally this allows us to upper bound the Handelman rank of any  perfect graph $G$ by its maximum clique size, with equality when $G$ is vertex-transitive (Proposition \ref{propperfect}).

Second we observe a simple identity for square-free polynomials (Lemma \ref{lemp2}), which can be used  to relate the algebraic operation of setting a variable to 0 (resp. to  1) to the graph operation of deleting a node (resp., deleting a node and its neighbours).
This technique permits to relate the Handelman rank with structural properties of graphs and can be applied to show the upper bounds and to deal e.g. with odd cycles and odd wheels.

In addition, for the maximum cut problem,  we  clarify how  the Handelman hierarchy applies to the formulation (\ref{popmaxcut}) and show that it can be reformulated as optimization over a polytope defined by an explicit subset of valid inequalities for the cut polytope; as an application we find again 
several results of \cite{PH11,PH12} (see Section~\ref{conclude}).

More specifically the paper is organized as follows.
In Section 2 we present  some preliminary results about square-free polynomials and the Handelman hierarchy.
In particular we prove  the error bound from Theorem \ref{theoPH} (for polynomials of arbitrary degree) and
we introduce  the
Handelman hierarchy for the maximum stable set problem.
Section 3 contains our new results. In Section \ref{sec31}
we show a relation to fractional clique coverings and  we show that the Handelman bound of order 2 is equal to the fractional stability number.
Section \ref{sec32} contains the  two new upper   bounds for the Handelman rank, in Section \ref{sec33}  we determine  the
Handelman rank of several  classes of graphs, and in Section \ref{sec34} we study the behaviour of the Handelman rank under some graph operations like edge deletion and clique sums.
In
Section 4 we point out links to the linear or semidefinite programming  hierarchies of
Sherali-Adams, Lasserre, Lov\'asz-Schrijver, and de Klerk-Pasechnik.
In  Section 5 we give an explicit formulation for the Handelman hierarchy applied to the maximum cut problem in terms of valid   inequalities of the cut polytope.

\subsection{Notation}

For an integer $n\ge 1$, we set $[n]:=\{1,2\dots,n\}$.
Given a finite set $V$ and an integer $t$, $\PP(V)$ denotes the collection of all
subsets of $V$,
$\PP_t(V):=\{I\subseteq V: |I|\le t\}$, and
$\PP_{=t}(V):=\{I\subseteq V: |I|=t\}$.
The support of $x\in \oR^n$ is the set $\{i\in [n]: x_i \ne 0\}$.
For $x\in \oR^n$ and $S\subseteq [n]$, $x(S)=\sum_{i\in S}x_i$.
We let $e$ denote the all-ones vector in $\oR^n$ and
$e_1,\dots,e_n$ denote the standard unit vectors in $\oR^n$.  For a subset $I\subseteq [n]$,
 $\chi^I\in \{0,1\}^n$ denotes its
characteristic vector.
The
space of symmetric $n\times n$ matrices is denoted as
$\mathcal{S}_n$. A matrix $A\in \SSS_n$ is positive semidefinite (resp., copositive) if  $x^TAx\ge 0$ for all $x\in \oR^n$ (resp., $x^TAx\ge 0$ for all $x\ge 0$).
Then, $\mathcal{S}_n^+$ denotes the positive semidefinite cone, consisting of all  positive semidefinite matrices in $\mathcal S_n$, and $\CC_n$ is the copositive cone, consisting of all copositive matrices.

 Let
$\oR[x]=\oR[x_1,\dots,x_n]$ denote the ring of multivariate
 polynomials in $n$ variables with real coefficients.
Monomials in $\oR[x]$ are denoted as
$x^\alpha=x_1^{\alpha_1}\cdots x_n^{\alpha_n}$ for $\alpha\in \oN^n$, with degree
$|\alpha|:=\sum_{i=1}^{n}\alpha_i$.
For a polynomial
$p=\sum_{\alpha\in\oN^n} p_\alpha x^\alpha$, its degree is defined as ${\rm
deg}(p):=\max_{\alpha|p_\alpha\ne0}|\alpha|$. For an integer $t$, $\oR[x]_t$ denotes the
subspace of   polynomials with degree at most
$t$. The monomial $x^\alpha$ is said to be {\em square-free} (aka {\em multilinear}) if $\alpha\in \{0,1\}^n$ and a polynomial $p$ is
{\em square-free} if all its monomials are square-free.
For $I\subseteq [n]$, we use the
notation $x^I=\prod_{i\in I}x_i$. Hence a square-free polynomial can be written as $\sum_{I\subseteq [n]} p_I x^I$.
 Given a subset
$S\subseteq\oR^n$, we say that $p\in\oR[x]$ is positive (resp.,
nonnegative) on $S$ when $p(x)>0$ (resp., $p(x)\ge0$) for all $x\in
S$. Given $g_1,\dots,g_m\in\oR[x]$ and $s\in\oN^m$, we often
use the notation $g^s=g_1^{s_1}\cdots g_m^{s_m}$, with $g^0=1$.
The ideal generated by  a set of polynomials $g_1,\ldots,g_m\in\oR[x]$ is the set, denoted as $\langle g_1,\ldots,g_m\rangle$, consisting  of all polynomials of the form $\sum_{j=1}^m u_jg_j$ where $u_j\in \oR[x]$.

Given a graph $G = (V,E)$, $\overline{G}=(V,\overline{E})$
denotes its complementary graph whose edges are the pairs of distinct nodes
$i,j\in V (G)$ with $ij\notin E$. Throughout we also set
$V=V(G)$, $E=E(G)$ and  we often assume $V(G)=[n]$. $K_n$ denotes the complete graph and  $C_n$ the
circuit on $n$ nodes. A set $S\subseteq V$ is stable (or independent) if no two distinct nodes of $S$ are adjacent in $G$ and a  clique in $G$ is a set of pairwise adjacent nodes. The maximum cardinality of a stable set (resp., clique) in $G$ is denoted by $\alpha(G)$ (resp., $\omega(G)$); thus $\omega(G)=\alpha(\overline G)$. The chromatic number $\chi(G)$ is the minimum number of colors needed to color the nodes of $G$ in such a way that adjacent nodes receive distinct colors. For a node $i\in V$,  $G-i$ denotes
the graph obtained by deleting node $i$ from $G$, and $G\ominus
i$ denotes the graph obtained from $G$ by removing $i$ as well
as the set $N(i)$ of its neighbours. For $U\subseteq V$, $G\backslash U$ denotes the graph obtained by deleting all nodes of $U$. For an
edge $e\in E$, let $G\backslash e$ denote the graph obtained by
deleting edge $e$ from $G$, and let $G/e$ denote the graph
obtained from $G$ by contracting edge $e$.
Consider two graphs $G_1=(V_1,E_1)$ and $G_2=(V_2,E_2)$ such that $V_1\cap V_2$ is a clique of cardinality $t$ in both $G_1$ and $G_2$. Then the graph $G=(V_1\cup V_2, E_1\cup E_2)$ is called the {\em clique $t$-sum} of $G_1$ and $G_2$.

\section{Preliminaries}

\subsection{Maximization of square-free polynomials over the hypercube}

In this section we group some observations about the Handelman hierarchy when it is applied to the problem of maximizing
 a square-free polynomial $p$ over the hypercube:
 $$\pmax=\max_{x\in [0,1]^n} p(x).$$
 In what follows we let $\II$ denote the ideal generated by the polynomials
$x_i^2-x_i$ for $i\in [n]$.
 Using the description of the hypercube by the  inequalities: $x_i\ge 0, 1-x_i\ge 0$ for $i\in [n]$, the corresponding Handelman set of order $t$ reads:
 \begin{equation}\label{eqsetH}
\mathcal{H}_t =\left\{\sum_{\alpha,\beta\in\oN^n: |\alpha+\beta|\le t} c_{\alpha,\beta} x^\alpha (1-x)^\beta: c_{\alpha,\beta}\ge 0\right\}.
\end{equation}
We also consider the following subset  consisting of all square-free polynomials in $\HH_t$   involving only terms which  do not lie in the ideal $\II$:
\begin{equation}\label{eqsettiH}
H_t:=\left\{\sum_{T\in \PP_t(V), I\subseteq T} c_{T,I} x^I (1-x)^{T\setminus I}: c_{T,I}\ge 0\right\}.
\end{equation}
Clearly, in the definition of $H_t$, we can restrict without
loss of generality to sets $T\in \PP_{=t}(V)$. Indeed, if
$T<t$, pick an element $k\in V\setminus T$ and elevate the
degree of $x^I (1-x)^{T\setminus I}$ by writing  $x^I
(1-x)^{T\setminus I}= x^{I\cup \{k\}}(1-x)^{T\setminus I} +
x^I(1-x)^{(T\setminus I)\cup \{k\}}$.

By construction, the Handelman bound $\phan^{(t)}$ for the maximum value $\pmax$ of $p$ over $[0,1]^n$ is defined using the set $\HH_t$ in (\ref{eqsetH}). We now show that it can alternatively be defined using  the subset $H_t$ in (\ref{eqsettiH}). 

 \begin{proposition} \label{prophan}
 Let $p\in\oR[x]$ be a square-free polynomial. For any integer $t\ge 1$,
 $$\phan^{(t)} := \inf\{\lambda: \lambda-p \in \HH_t\}=\inf \{\lambda: \lambda-p \in H_t\}.$$
 \end{proposition}

This result follows directly from  Lemma \ref{lemhan} below, whose proof relies on  the following Lemmas \ref{lemh1} and \ref{lemh2}.

\begin{lemma}\label{lemh1}
If $p$ is a square-free
polynomial and $p\in \II$, then $p=0$.
\end{lemma}

\begin{proof}
We use induction on the number $n$ of variables. In the case
$n=1$, we have that $p=p_0+p_1x_1 = f_1\cdot(x_1-x_1^2)$, which
implies $f_1=0$ and thus $p=0$ by looking at the degrees of
both sides. Suppose now that the result holds for $n=k-1$. Let
$p$ be a square-free polynomial in $k$ variables lying in the
ideal $\II$. We can write $p$ as $p(x)=p_0(\ux)+x_k p_1(\ux)$,
where $p_0$, $p_1$ are square-free in the $k-1$ variables
$\ux=(x_1,\cdots,x_{k-1})$. Say, $p_0+x_kp_1= p=\sum_{i=1}^k
f_i\cdot (x_i-x_i^2)$ for some polynomials $f_i$. By setting
$x_k=0$ we get: $p_0(\ux)=\sum_{i=1}^{k-1}f_i(\ux,0)
(x_i-x_i^2)$. As $p_0$ is square-free, we deduce using the
induction assumption that $p_0=0$. Next,   by setting $x_k=1$,
we get: $p_1(\ux) =\sum_{i=1}^{k-1} f_i(\ux,1)(x_i-x_i^2)$. As
$p_1$ is square-free we deduce from the induction assumption
that $p_1=0$. Thus we have shown that $p=0$.
\qed\end{proof}

\begin{lemma}\label{lemh2}
Given $\alpha, \beta\in \oN^n$, let $I=\{i\in [n]:\alpha_i\ge 1\}$ and $J=\{i\in [n]: \beta_i\ge 1\}$ denote their supports.
\begin{itemize}
\item[(i)]  If $I\cap J\ne \emptyset$ then  $x^\alpha
    (1-x)^\beta $ belongs to $\II$.
\item[(ii)]  If $I\cap J=\emptyset$ then $x^\alpha
    (1-x)^\beta - x^I(1-x)^J $ belongs to  $\II$.
\end{itemize}
\end{lemma}
\begin{proof}
(i) Say, $1\in I\cap J$. Then $x_1(1-x_1)$ is a factor of
$x^\alpha (1-x)^\beta$ and thus  $x^\alpha (1-x)^\beta\in \II$.\\
(ii) The proof is based on using iteratively  the following identities,
for any $k\ge 2$:
$$x_i^{k}-x_i = (x_i^2-x_i)(x_i^{k-2}+\cdots +x_i+1)\in \II,$$
$$(1-x_i)^k- (1-x_i)=- x_i(1-x_i) ((1-x_i)^{k-2}+\cdots + (1-x_i)+1)\in \II.$$
Indeed, $x^\alpha(1-x)^\beta-x^I(1-x)^J=
(x_1^{\alpha_1}-x_1)\ux^{\underline{\alpha}}(1-x)^{\beta} +
x_1( \ux^{\underline{\alpha}}(1-x)^\beta - \ux^{I\setminus
\{1\}}(1-x)^J),$ setting $\ux=(x_2,\cdots,x_n)$ and
$\underline{\alpha}=(\alpha_2,\cdots,\alpha_n)$.
\qed\end{proof}

\begin{lemma}\label{lemhan}
Let $p$ be a square-free  polynomial and $t\ge 1$ an integer.
The following assertions are equivalent.
\begin{itemize}
\item[(i)] $p\in \HH_t$.
\item[(ii)] $p\in  H_t +\II$.
\item[(iii)] $p\in  H_t$.
\end{itemize}
\end{lemma}

\begin{proof}
(i) $\Longrightarrow$ (ii):  Say, $p=\sum_{A} c_{\alpha,\beta}x^\alpha
(1-x)^{\beta}$ where $c_{\alpha,\beta}\ge 0$. Group in the polynomial
$p_0=\sum_{A_0} c_{\alpha,\beta} x^\alpha (1-x)^\beta$ all the
terms of $p$ where the supports of $\alpha$ and $\beta$ are not
disjoint. Let  $S_\alpha$  denote the support of $\alpha$.
Then,  we have:
$$p =p_0+ \sum_{A\setminus A_0} c_{\alpha,\beta}( x^\alpha(1-x)^{\beta} -x^{S_{\alpha}} (1-x)^{S_{\beta}}) + \sum_{A\setminus A_0} c_{\alpha,\beta} x^{S_{\alpha}} (1-x)^{S_{\beta}}.$$
By Lemma \ref{lemh2},  the first two  sums lie in $\II$ and the
last sum lies in $ H_t$ and thus $p\in H_t+\II$.\\
The implication (ii) $\Longrightarrow $ (iii) follows from Lemma \ref{lemh1}
and  (iii) $\Longrightarrow$ (i) follows from the inclusion
 $ H_t\subseteq \HH_t$.
\qed\end{proof}

 As an application of
Lemma \ref{lemh1}, we also find the following representation for square-free polynomials, which corresponds to the fact that the polynomials $\{x^I(1-x)^{[n]\setminus I}: I\subseteq [n]\}$ form a basis of the vector space of square-free polynomials.

\begin{corollary}\label{corhan}
Any square-free  polynomial $p$ can
be written as
\begin{equation}\label{eqpV}
p=\sum_{I\subseteq [n]} p(\chi^I)
x^I(1-x)^{[n]\setminus I}.
\end{equation}
\noindent Therefore, if $p(x)\ge 0$ for
all $x\in \{0,1\}^n$, then  $p\in H_n$.
\end{corollary}

\begin{proof}
The polynomial $p-\sum_{I\subseteq [n]} p(\chi^I)
x^I(1-x)^{[n]\setminus I}$ is
square-free and vanishes on $\{0,1\}^n$. Hence it belongs to
the ideal $\II$ and thus it is identically zero, by Lemma \ref{lemh1}.
\qed\end{proof}

In particular, as the polynomial $\pmax-p$ is nonnegative on the hypercube, we find again the convergence: $\phan^{(n)} =\pmax$ of the Handelman hierarchy in $n$ steps, when $p$ is square-free. We mention another application which we will use later in the paper.

\begin{lemma}\label{lemp2}
Let $f$ be a square-free polynomial in $n$ variables $x=(x_1,\ldots,x_n)=(\underline{x},x_n)$, setting
$\underline{x}=(x_1,x_2\dots,x_{n-1})$. Then, one has
$$f(x)=(1-x_n)f(\underline{x},0)+x_nf(\underline{x},1).$$
\end{lemma}

\begin{proof}
Using (\ref{eqpV}) (and splitting the sum into two sums depending whether $I$ contains $n$ or not),
we can write $f(x)$ as
$f(x)= x_n f_1(\underline{x}) + (1-x_n) f_2(\underline{x})$.
By evaluating $f$ at $(\underline{x},0)$ and  $(\underline{x},1)$, we obtain  that  $f(\underline{x},0)=f_2(\underline{x})$ and
 $f(\underline{x},1)=f_1(\underline{x})$, which gives the result.
\qed\end{proof}

\subsection{Error bound of Handelman hierarchy}\label{secerror}

We now extend the result of Theorem \ref{theoPH} analyzing  the Handelman bound of any order $t\le n$ to polynomials of arbitrary degree.

\begin{theorem}
Let $p=\sum_{J\subseteq [n]}p_Jx^J$ be a square-free
polynomial with $p(0)=0$.
 For any integer $t$ satisfying $\dd(p)\le t\le n$, we have
$$
\phan^{(t)}\le {n\over t} \pmax +\sum_{J\subseteq [n]: |J|\ge 2, p_J>0} p_J\lambda_J,
$$
setting $$\lambda_J= \left( {n-1\choose t-1}-{n-|J|\choose t-|J|}\right)/{n-1\choose t-1}\ \  \text{ for } J\subseteq [n].$$
Hence, if  $p_J\le 0$ for all $J\subseteq [n]$ with $|J|\ge 2$,  then  $$\phan^{(t)}\le {n\over t}\pmax.$$
\end{theorem}

\begin{proof}
The proof is along the same  lines as the proof of \cite[Proposition 3.2]{PH12} and uses
 the following `combinatorial' Bernstein approximation of $p$, defined as
$$B_t(p):=\sum_{T\in\PP_{=t}([n])}\sum_{I\subseteq T}p(\chi^I)x^I(1-x)^{T\backslash I}.$$
One can check that
$$B_t(x^J)=\sum_{ T\in \PP_{=t}([n]): J\subseteq T}\sum_{I: J\subseteq I\subseteq T}x^I(1-x)^{T\backslash I}=\sum_{T\in\PP_{=t}([n]): J\subseteq T}x^J={n-|J|\choose t-|J|}x^J$$
for any $J\subseteq [n]$. Hence, the Bernstein approximation of $p=\sum_{J\subseteq
[n]}p_Jx^J$ reads
\begin{equation}\label{cba}
B_t(p)=\sum_{J:J\subseteq
[n],|J|\le t}p_J{n-|J|\choose t-|J|}x^J.
\end{equation}
Now we divide throughout by $n-1\choose t-1$ and add to both sides of (\ref{cba}) the quantity
$\sum_J p_J\lambda_Jx^J$
to get
\begin{equation*}\label{cba1}
{B_t(p)\over {n-1\choose t-1}}+ \sum_J p_J\lambda_Jx^J =p.
\end{equation*}
As $B_t(1)={n\choose t}={n\over t}{n-1\choose t-1}$, this gives  ${n\over t} \pmax= {B_t(\pmax)\over {n-1\choose t-1}}$ and thus we obtain
\begin{equation}\label{eqBB}
{n\over t}\pmax -p = {B_t(\pmax-p)\over {n-1\choose t-1}} -\sum_J\lambda_Jp_Jx^J.
\end{equation}
As the polynomial $\pmax-p$ is nonnegative over $\{0,1\}^n$, it follows from the definition of the Bernstein operator that
$$B_t(\pmax-p)=\sum_{T\in \PP_{=t}([n])}\sum_{I\subseteq T}(p_{\max}-p(\chi^I))x^I(1-x)^{T\backslash I} \in H_t.$$
As $\lambda_J\ge 0$ for all $J$, after moving the terms $p_J\lambda_Jx^J$ with $p_J>0$ to the left hand side of (\ref{eqBB}), we obtain the claimed inequalities.
\qed\end{proof}

\subsection{The maximum stable set problem}

Let $G=(V,E)$ be a graph and let $w\in \oR^V_+$ be weights assigned to the nodes of $G$.
The
{\em maximum stable set problem} is to determine the maximum
weight $w(S)=\sum_{i\in S}w_i$ of a stable set $S$ in $G$,
called the weighted {\em stability number} of $(G,w)$ and denoted as $\alpha(G,w)$.
Let $\ST(G)$ denote the
polytope in $\oR^V$, defined as the convex hull of the
characteristic vectors of the stable sets of $G$:
\begin{equation*}\label{stp}
\ST(G):={\rm{conv}}\{\chi^S : S\subseteq V,\ \ \text{$S$ is a stable set in $G$}\},
\end{equation*}
called the {\em stable set polytope} of G. Hence, computing
$\alpha(G,w)$ is a linear optimization problem over the stable set
polytope: $$\alpha(G,w)=\max_{x\in \ST(G)}\sum_{i\in V}w_ix_i.$$
It is well known that computing  $\alpha(G,w)$ is an NP-hard problem, already in the unweighted case when $w=e$ \cite{Karp72}.
An obvious
linear relaxation of $\ST(G)$ is the {\em fractional stable set
polytope} $\FR(G)$,   defined as
\begin{equation*}\label{frg}\FR(G):=\{x\in
\oR^V: x\ge 0,\ x_i+x_j\le 1\ \forall ij\in E\}.
\end{equation*}
By maximizing the linear objective function $w^Tx$  over $\FR(G)$ we obtain  an upper bound
for the stability number:
\begin{equation}\label{fracst}
\alpha^*(G,w):=\max_{x\in \FR(G)}\sum_{i\in V}w_ix_i,
\end{equation}
called the {\em fractional stability number}.

\medskip
We now consider another  formulation for $\alpha(G,w)$ obtained by maximizing a suitable quadratic polynomial over the hypercube.
Given  node weights $w\in \oR^V_+$, we consider edge weights $w_{ij}$ for the edges of $G$ satisfying the condition
\begin{equation}\label{wedge0}
w_{ij}\ge  \min\{w_i,w_j\}\ \ \text{ for all edges } ij\in E.
\end{equation}
For some of our results we will need to make a stronger assumption on the edge weights:
\begin{equation}\label{wedge1}
w_{ij}\ge \max\{w_i,w_j\} \ \ \text{ for all edges } ij \in E,
\end{equation}
More precisely, we will use (\ref{wedge1}) in Sections 3.2.2, 3.2.3, 3.3.1 and 3.3.2.
In the weighted case,
unless specified otherwise,  we will assume that the edge
weights  satisfy the weakest condition (\ref{wedge0}). In the
unweighted case (i.e. $w_i=1$ for all nodes $i\in V$), we simply define $w_{ij}=1$ for all edges $ij\in E$. Once
the edge weights are specified we define the (square-free
quadratic) polynomials
\begin{equation*}\label{pGw}
p_{G,w}:=\sum_{i\in V} w_ix_i -\sum_{ij\in E}w_{ij} x_ix_j,
\end{equation*}
\begin{equation}\label{fGw}
f_{G,w}:=\alpha(G,w)-p_{G,w}=\alpha(G,w)-\sum_{i\in V} w_ix_i+\sum_{ij\in E}w_{ij}x_ix_j.
\end{equation}
In the unweighted case $p_{G,w}$ is the polynomial used earlier in the formulation (\ref{popalpha}).

In this paper we are interested in establishing positivity certificates for the polynomial $f_{G,w}$ and in understanding what is the smallest integer $t$ for which $f_{G,w}$ belongs to the Handelman set $\HH_t$, see Definition \ref{defrkH} below. It is clear that we get stronger positivity certificates if we can show that $f_{G,w}\in \HH_t$ for lower values of the edge weights. This motivates our distinction between the above two conditions (\ref{wedge0}) and (\ref{wedge1}) on the edge weights.

Park and Hong \cite{PH12} give the following reformulation for the maximum stable set problem
(choosing $w_{ij}=\max\{w_i,w_j\}$ for the edge weights), we give a proof for completeness.

\begin{proposition}
Given node weights $w\in \oR^V_+$ and edge weights satisfying (\ref{wedge0}), the maximum stable set
problem can be reformulated as
\begin{equation}\label{alpha}
\alpha(G,w)=\max_{x\in [0,1]^V} p_{G,w}(x)= \max_{x\in \{0,1\}^n} p_{G,w}(x).
\end{equation}
\end{proposition}

\begin{proof}
As $p_{G,w}$ is square-free,  it takes the same maximum value on $[0,1]^n$ and $\{0,1\}^n$.
Clearly, the maximum value over  $\{0,1\}^n$ is at least $\alpha(G,w)$ since $p_{G,w}$ evaluated at the characteristic vector of a maximum weight stable set  is equal to $\alpha(G,w)$.
It suffices now to observe that the maximum value of $p_{G,w}$ over $\{0,1\}^n$ is attained at the characteristic vector of a stable set.
Indeed, for $S\subseteq V$,
$p_{G,w}(\chi^S)=\sum_{i\in S}w_i-\sum_{ij\in E: i,j\in S}w_{ij}$.
If  $ij$ is an edge contained in $S$ with  $w_j\ge w_i$, then
$p_{G,w}(\chi^{S\setminus \{i\}}) -p_{G,w}(\chi^S) \ge w_{ij}-w_i\ge 0$.
Hence we can replace
$S$ by $S\backslash \{i\}$  without
decreasing the objective value $p_{G,w}$. Iterating,
we obtain that the maximum value of $p$ over $\{0,1\}^n$  is
attained  at a stable set.
\qed\end{proof}

By Proposition \ref{prophan}, the Handelman bound of order $t$ for  problem (\ref{alpha}) reads:
\begin{equation}\label{phantgw}
\phan^{(t)}(G,w):=\inf
\{\lambda: \lambda -p_{G,w} \in  H_t\}
\end{equation}
and, by Theorem \ref{theoPH}, it satisfies the inequality: $\phan^{(t)}(G,w)\le {n\over t} \alpha(G,w)$.

\begin{definition}\label{defrkH}
We let  $\rk(G,w)$ denote
the smallest integer $t$ for which
$\phan^{(t)}(G,w)=\alpha(G,w)$, called the {\em Handelman rank}
of the weighted graph $(G,w)$.
Equivalently, $\rk(G,w)$ is the smallest integer $t$ for which $f_{G,w}$ belongs to the Handelan set $\HH_t$.
\end{definition}

For the all-ones weight function $w=e$ (i.e., the unweighted
case) we  omit the subscript $w$ and simply write $p_G$, $f_G$,
$\phan^{(t)}(G)$, and $\rk(G)$.

\medskip
If $G$ has no edge then $\rk(G,w)=1$, since
$\alpha(G,w)-p_{G,w}=\sum_{i\in V} w_i(1-x_i)\in H_1$, and the Handelman rank is at least 2 if $G$ has at least one edge.
As another example, it follows from Corollary \ref{corhan} that, for the complete graph $K_n$, the polynomial $f_{K_n}$ belongs to $H_n$.

\begin{lemma}\label{lemC}\cite{PH12}
The polynomial
$f_{K_n}=\alpha(K_n)-p_{K_n} = 1-\sum_{i=1}^n x_i +\sum_{1\le i<j\le n} x_ix_j$
belongs to $H_n$.
\end{lemma}

\section{The Handelman hierarchy for the maximum stable set problem}

\subsection{Links to clique covers}\label{sec31}

In this section we show an upper bound for the Handelman bound
in terms of fractional clique covers, and we characterize the
graphs with Handelman rank at most 2.

First, we introduce {fractional clique covers.} Let $(G,w)$ be
a weighted graph. A  {\em fractional clique cover}  of
$(G,w)$ is a collection of cliques $C$  of $G$ together with
scalars $\lambda_C\ge 0$ satisfying $\sum_{C}\lambda_C \chi^C
=w$. Then the minimum value of $\sum_C \lambda_C$ is known as
the weighted {\em fractional chromatic number} of $\overline
G$:
\begin{equation}\label{fracchro}
\chi^*(\overline G,w) =\min\left\{ \sum_C \lambda_C: \sum_{C } \lambda_C \chi^C = w,\ \lambda _C\ge 0 \ \forall C \text{ clique of } G\right\}.
\end{equation}
Note  that if  in addition we  require the  $\lambda_C$'s to be integer valued in (\ref{fracchro}) then we obtain the chromatic number
$\chi(\overline G,w)$.
Restricting to covers by cliques of size at most some given
integer $t\ge 1$, we can define the parameter
\begin{equation}\label{fracchrot}
\rho_t(G,w):=\min\left\{ \sum_C \lambda_C: \sum_{C } \lambda_C \chi^C = w,\ \lambda _C\ge 0 \ \forall C \text{ clique of } G \text{ with } |C|\le t\right\},
\end{equation}
which we call the {\em fractional $t$-clique cover number} of
$(G,w)$.  Thus
\begin{equation*}
\rho_t(G,w)=\chi^*(\overline G,w)\ \text{
if } \ t\ge \omega(G),
\end{equation*}
where $\omega(G)$ denotes  the largest size of a clique in $G$. In
addition,
$$\rho_t(G,w)\geq \chi^*(\overline G,w)\geq \alpha(G,w).$$
 As is well known,   in relation (\ref{fracchro}) one
can relax without loss of generality the equality $\sum_C
\lambda_C\chi^C=w$ to the inequality $\sum_C\lambda_C\chi^C \ge
w$. This extends to the fractional clique cover number.  We include a short argument for clarity.

\begin{lemma}
The parameter $\rho_t(G,w)$ from (\ref{fracchrot}) is equal to the optimal value of the following program:
\begin{equation}\label{fracchrotnew}
\min\left\{ \sum_C \lambda_C:
\sum_{C } \lambda_C \chi^C \ge w,\ \lambda _C\ge 0 \ \forall C
\text{ clique of } G \text{ with } |C|\le t\right\}.
\end{equation}
\end{lemma}

\begin{proof}
Comparing (\ref{fracchrot}) and (\ref{fracchrotnew}), one only needs to show that the optimal value of (\ref{fracchrotnew}) is at least $\rho_t(G,w)$. The argument is easier by looking at the dual linear programs.
The  dual of (\ref{fracchrot}) reads
\begin{equation}\label{fracdual}
\max\left\{ \sum_{i\in V}w_ix_i:\sum_{i\in C}x_i\le 1\ \forall C
\text{ clique of } G \text{ with } |C|\le t \right\}
\end{equation}
and the dual of (\ref{fracchrotnew}) reads
\begin{equation}\label{fracdual1}
\max\left\{ \sum_{i\in V}w_ix_i:\sum_{i\in C}x_i\le 1\ \forall C
\text{ clique of } G \text{ with } |C|\le t,\  x_i\ge0\ \forall i\in V \right\}.
\end{equation}
Suppose $x^*\in\oR^n$ is an optimal solution of the program (\ref{fracdual}). Then define $y\in\oR^n$ by setting  $y_i=x_i$ if  $x_i\ge 0$ and $y_i=0$ otherwise. Then, $\sum_i w_ix_i^*\le \sum_i w_i y_i$. It suffices now to show that $y$ is
feasible for the program (\ref{fracdual1}).
For this, pick a clique $C$
with  $|C|\le t$, and let $C^*$ denote the subset of $C$ consisting of all elements $i\in C$ with $x_i^*\ge 0$. Then $C^*$ is again a clique with $|C^*|\le t$ and thus $\sum_{i\in C^*} y_i = \sum_{i\in C^*}x_i^* \le 1$, which concludes the proof. \qed
\end{proof}

For
$t=2$, $\rho_2(G,w)$ is the fractional edge cover number, which coincides with the fractional stability number $\alpha^*(G,w)$ of (\ref{fracst}).
Indeed, for $t=2$,  the  program (\ref{fracst}) coincides  with (\ref{fracdual1}) which is the dual of the program (\ref{fracchrotnew}) defining
$\rho_2(G,w)$.

\begin{proposition}\label{lem1}
Consider a weighted graph $(G,w)$ with edge weights satisfying (\ref{wedge0}). For any integer $t\ge 2$,
 $$\rho_t(G,w)  -p_{G,w} \in H_{t} \ \text{ and } \ \phan^{(t)}(G,w)\le \rho_t(G,w).$$

\end{proposition}

\begin{proof}
Set $k=\rho_t(G,w)$. By definition (\ref{fracchrot}), there
exist scalars $\lambda_C\ge 0$ indexed by cliques $C$ of size
at most $t$ such that (a)  $\sum_C\lambda_C=k$, and (b)
$w=\sum_C\lambda_C\chi^C$,
 i.e., $w_i=\sum_{C: i\in
C}\lambda_C$ for all $i\in V$.  In particular, this implies
that (c) $\sum_{C: i,j\in C}\lambda_C\le \min\{w_i,w_j\}\le
w_{ij}$ for all $ij\in E$. Moreover, by taking
 the inner product of both sides of (b) with the vector $(x_1,\cdots,x_n)^T$, we get
$\sum_{i=1}^n w_i x_i=\sum_C \lambda_C x(C)$. Therefore,
\begin{eqnarray*}
k-p_{G,w}&=& \sum_C \lambda_C -\sum_{i\in V} w_ix_i +\sum_{ij\in E} w_{ij}x_ix_j \\
&=& \sum_C \lambda_C \left(1-\sum_{i\in C} x_i +\sum_{i<j: i,j\in C}x_ix_j\right) +\sum_{ij\in E} w_{ij} x_ix_j -\sum_C \lambda_C \sum_{i<j: i,j\in C}x_ix_j\\
&=& \sum_C \lambda_C f_C +\sum_{ij\in E} w_{ij} x_ix_j -\sum_C \lambda_C \sum_{i<j: i,j\in C}x_ix_j,
\end{eqnarray*}
setting $f_C=1-\sum_{i\in C}x_i +\sum_{i<j: i,j\in C}x_ix_j$.
By Lemma \ref{lemC}, each $f_C$ lies in  $H_t$ and thus the first sum lies in $H_t$.
We now consider the remaining part:
$$\sum_{ij\in E}w_{ij} x_ix_j -\sum_C \lambda_C \sum_{i<j: i,j\in
C}x_ix_j = \sum_{ij\in E} x_ix_j\left(w_{ij} -\sum_{C: i,j\in
C}\lambda_C\right),$$ which belongs to $H_2$  since the scalars $w_{ij}-\sum_{C:
i,j\in C}\lambda_C$ are nonnegative by (c).
Thus we have shown that $k-p_{G,w}\in H_t$, which gives directly $\phan^{(t)}(G,w)\le k$.
\qed\end{proof}
Next, we show that equality $\phan^{(t)}(G,w)=\rho_t(G,w)$ holds
for $t=2$. Note that for $t\ge 3$, the strict inequality $\phan^{(t)}(G,w)<\rho_t(G,w)$ is possible. For instance, for the odd circuit $C_{2n+1}$, $\phan^{(3)}(C_{2n+1})=\alpha(C_{2n+1})<\rho_3(C_{2n+1})=\alpha^*(C_{2n+1})$ holds (see Proposition \ref{lemoddcircuit} below).
\begin{theorem}\label{thmrho2g}
Consider a  weighted graph $(G,w)$ with edge weights satisfying (\ref{wedge0}). Then,
$\phan^{(2)}(G,w)=\rho_2(G,w)$.
\end{theorem}

\begin{proof}
Set $k=\phan^{(2)}(G,w)$. In what follows we construct a
fractional 2-clique covering of $(G,w)$ of value $k$, which shows the inequality $\rho_2(G,w)\le \phan^{(2)}(G,w)$ and concludes the proof.
 By
assumption, the polynomial $k -p_{G,w}$ belongs to $H_2$ and thus
 has a decomposition:
\begin{equation}\label{eq1}
k -p_{G,w}=  \sum_{ij\in E_n} a_{ij} (1-x_i)(1-x_j) + b_{ij} x_i(1-x_j) +c_{ij}x_j(1-x_i) +d_{ij} x_ix_j
\end{equation}
where all scalars $a_{ij},b_{ij},c_{ij},d_{ij}\ge 0$ and $E_n$
denotes the set of ordered pairs $ij$ with $1\le i<j\le n$. By
evaluating the coefficients of the monomials $1$, $x_i$ and
$x_ix_j$ we get the relations:
\begin{equation*}\label{eq10}
k=\sum_{ij\in E_n} a_{ij},
\end{equation*}

\begin{equation*}\label{eq1i}
-w_i= -\sum_{j: j>i} a_{ij} -\sum_{j: j<i} a_{ji}+\sum_{j: j>i} b_{ij} + \sum_{j: j<i} c_{ji} \ \ \ \text{ for any }  i\in V,
\end{equation*}

\begin{equation}\label{eq1ij}
a_{ij}-b_{ij}-c_{ij} +d_{ij}=\left\{\begin{array}{ll} w_{ij} & \text{ if } ij \in E \\
0 & \text{ otherwise}. \end{array}\right. \ \ \ \text{ for any pair } ij\in E_n.
\end{equation}
First we observe that we can find another decomposition of
$k-p_{G,w}$, of the  form \eqref{eq2} below, which involves
quadratic terms only for the edges of $G$ but has additional linear terms.
For any pair $ij\in E_n$, set
$$f_{ij}= a_{ij}(1-x_i)(1-x_j)+b_{ij}x_i(1-x_j)+c_{ij}x_j(1-x_i)+d_{ij}x_ix_j$$
so that the decomposition  (\ref{eq1}) reads:
$k-p_{G,w}= \sum _{ij\in E_n } f_{ij}.$
We now show that, for any $ij\in E_n\setminus E$, the polynomial $f_{ij}$ belongs to $H_1$.
Indeed, pick a pair $ij$
which is not an edge. By (\ref{eq1ij}), we have:
$d_{ij}=b_{ij}+c_{ij}-a_{ij}$, so that we can rewrite $f_{ij}$ as
$$f_{ij}= x_i(b_{ij}-a_{ij})+x_j(c_{ij}-a_{ij}) +a_{ij}.$$
We distinguish several cases:\\
$\bullet$ If $b_{ij}-a_{ij}\ge 0$ and $c_{ij}-a_{ij}\ge 0$ then
we get a representation in $H_1$ for $f_{ij}$.
\\
$\bullet$ If $b_{ij}-a_{ij}\le 0$ and $c_{ij}-a_{ij}\ge 0$ then
rewrite $f_{ij}$ as:
$$f_{ij}= (1-x_i)(a_{ij}-b_{ij}) + x_j(c_{ij}-a_{ij}) + b_{ij}\in H_1.$$
\noindent $\bullet$ Analogously if $b_{ij}-a_{ij}\ge 0$ and $c_{ij}-a_{ij}\le 0$.\\
$\bullet$ If $b_{ij}-a_{ij}\le 0$ and $c_{ij}-a_{ij}\le 0$ then
rewrite $f_{ij}$ as:
$$f_{ij}= (1-x_i)(a_{ij}-b_{ij}) +(1-x_j)(a_{ij}-c_{ij}) + b_{ij}+c_{ij}-a_{ij}$$
which is again a representation in $H_1$ since $
b_{ij}+c_{ij}-a_{ij}=d_{ij}\ge 0$. Hence, we have shown $f_{ij}\in H_1$ for all  nonedges and thus
  we obtain a new
representation of $k-p_{G,w}$ of the form:
\begin{equation}\label{eq2}
k-p_{G,w}= \sum_{ij\in E} a_{ij} (1-x_i)(1-x_j) + b_{ij} x_i(1-x_j) +c_{ij}x_j(1-x_i)+d_{ij} x_ix_j +\sum_{i\in V} f_i x_i+g_i(1-x_i),
\end{equation}
where all coefficients $a_{ij},b_{ij},c_{ij},d_{ij},f_i,g_i$ are nonnegative scalars.
 Then, we obtain:
\begin{equation}\label{eq20}
k=\sum_{ij\in E}a_{ij}+\sum_{i\in V}g_i,
\end{equation}
and for all $i\in V$:
\begin{equation}\label{eq2i}
-w_i = -\sum_{j: j>i,ij\in E} a_{ij} -\sum_{j: j<i,ij\in E} a_{ji}+\sum_{j: j>i,ij\in E} b_{ij} +\sum_{j: j<i,ij\in E} c_{ji} + f_i-g_i.
\end{equation}
We now build a fractional clique cover. For this consider the
vector:
$$u= \sum_{ij\in E,i<j} a_{ij}\chi^{\{i,j\}}+\sum_{i\in V} g_i\chi^{\{i\}}.$$
We  check that $u_i\ge w_i$ for all $i\in V$.
For this fix $i$ and set $N=\{j: ij \in
E\}$. We have:
$$u_i= \sum_{j\in N:j>i} a_{ij} +\sum_{j\in N:j<i} a_{ji}+g_i.$$
Using (\ref{eq2i}) we get:
\begin{eqnarray*}
w_i &=& \sum_{j\in N:j>i} a_{ij}+\sum_{j\in N:j<i} a_{ji}-\sum_{j\in N: j>i} b_{ij} -\sum_{j\in N: j<i} c_{ji} -f_i+g_i.
\end{eqnarray*}
Thus $u_i\ge w_i$ is equivalent to
$$0\ge -\sum_{j\in N: j>i} b_{ij}-\sum_{j\in N: j<i} c_{ji}
-f_i.$$ It suffices now to observe that indeed $f_i\ge 0$, $
\sum_{j\in N: j>i} b_{ij}\ge 0$, and $\sum_{j\in N: j<i} c_{ji}
\ge 0.$
Hence  $u$ is a fractional 2-clique cover of $(G,w)$ with value
$\sum_{ij\in E}a_{ij}+\sum_{i\in V}g_i=k$ by (\ref{eq20}).
This implies that $\rho_2(G,w)\le k$ and concludes the proof.
\qed\end{proof}

\noindent
Now we can characterize the graphs with Handelman rank equal to
2.
\begin{corollary}\label{corrk2}
The Handelman bound of order 2 is exact if and only if there
is a fractional edge covering of value $\alpha (G,w)$, i.e.,
 $$\phan^{(2)}(G,w)=\alpha(G,w)\Longleftrightarrow \alpha(G,w)=\rho_2(G,w)\Longleftrightarrow \alpha^*(G,w)=\alpha(G,w).$$
\end{corollary}

It is well known that the equality $\alpha(G,w)=\alpha^*(G,w)$ holds for any node
weights $w\in \oR^V_+$ if and only if $G$ is bipartite \cite[Section 4]{Lov94}.
This implies that the Handelman rank of any weighted bipartite graph is at most 2, settling an open question  of Park and Hong
\cite{PH12} who proved the result in the unweighted case.

\begin{corollary}\label{corbip}
If $G$ is bipartite, then $\rk(G,w)\le 2$ for any node weights
 $w\in \oR^V_+$.
\end{corollary}

On the other hand, the Handelman hierarchy is sometimes exact
at order 2 for non-bipartite graphs, as the next example shows.

\begin{example}\label{exsmallrkH}
Let $G$ be the graph on $2t$ nodes obtained by taking the
clique sum
of $t$ copies of $K_{t+1}$ along a common clique $K_t$. Then $\alpha
(G)=t$, $\rho_2(G)=t$ (since one can cover all nodes by $t$ disjoint edges), and
thus the Handelman relaxation of order 2 is exact: $\rk(G) = 2$.
\end{example}

\subsection{Bounds for the Handelman rank}\label{sec32}

In this section, we show some lower and  upper bounds for the
Handelman rank of weighted graphs. The upper bounds hold when
assuming that the edge weights satisfy (\ref{wedge1}).

\subsubsection{Lower bound}
We start with the following lemma from \cite[Prop. 3.3]{PH12}  which we prove for completeness.

\begin{lemma}\label{lemp1}
Consider a square-free polynomial
$p(x)=a_0+\sum_{i\in[n]}a_ix_i+\sum_{I\subseteq
[n]:|I|\ge 2}a_Ix^I$. If $\lambda-p\in H_t$, then
$\lambda-a_0\ge\sum_{i\in[n]}a_i/t$.
\end{lemma}

\begin{proof}
Say, $\lambda-p=\sum_{T\in \PP_{=t}(V), I\subseteq T}
c_{I,T} x^I (1-x)^{T\setminus I}$ with $c_{I,T}\ge 0$.
Evaluating the constant term we find that
$$\lambda-a_0 =\sum_{T\in \PP_{=t}(V)} c_{\emptyset,T}.$$
Evaluating the coefficient of $x_i$ we get:
$$-a_i = \sum_{T\in \PP_{=t}(V): i\in T} \left(c_{\{i\},T} -c_{\emptyset,T}\right).$$
Summing up over all $i\in V=[n]$ gives:
$$-\sum_{i\in[n]}a_i= \sum_{i\in[n]}\sum_{T\in \PP_{=t}(V):i\in T} c_{\{i\},T} -\sum_{i\in[n]} \sum_{T\in \PP_{=t}(V):i\in T} c_{\emptyset,T}\ge -\sum_{T\in \PP_{=t}(V)}t c_{\emptyset,T}=-t(\lambda-a_0),$$
which implies $\lambda-a_0\ge \sum_{i\in[n]}a_i/t$.
\qed\end{proof}

Applying Lemma \ref{lemp1} to the polynomial $p_{G,w}$  we obtain the following lower bound on the Handelman rank.

\begin{proposition} \label{lembound2}
Consider a weighted graph $(G,w)$ where the edge weights
satisfy (\ref{wedge0}). Then,
 $\phan^{(t)}(G,w) \geq \frac{\sum_{i=1}^{n}w_i}{t}$. Therefore,
\begin{equation}\label{lowerboundrk}
\rk(G,w)\geq
\frac{\sum_{i=1}^{n}w_i}{\alpha(G,w)}.
\end{equation}
\end{proposition}

For the unweighted complete graph $G=K_n$, the lower bound is equal to $n$, which implies $\rk(K_n)\ge n$. Hence equality holds:  $\rk(K_n)=n$ and the lower bound is tight.

\subsubsection{The first upper bound}

First we show an upper bound for the Handelman rank
of a weighted graph $(G,w)$, in terms of parameters of the unweighted graph $G$.

\begin{theorem} \label{thmupperbound}
Consider a  weighted graph $(G,w)$ where the edge weights satisfy (\ref{wedge1}). Then,
\begin{equation}\label{upperboundrk}
\rk(G,w)\le
|V(G)|-\alpha(G)+1.\end{equation}
\end{theorem}

Note that the upper bund (\ref{upperboundrk})
is tight for the unweighted complete graph $K_n$.
The proof of Theorem  \ref{thmupperbound}  relies on Lemma \ref{thmvup} below  which will allow to use induction on the number of nodes.

In what follows we use the following notation: Given  a weighted graph $(G,w)$ and a subset $U\subseteq V$,  $(G\backslash U,w)$  denotes the weighted graph $G\backslash U$ where the node and edge weights are obtained from those of $G$ simply by restricting to nodes and edges of $G\backslash U$.

\begin{lemma} \label{thmvup}
Consider a  weighted graph $(G,w)$ where the edge weights satisfy (\ref{wedge1}). For any node $i\in V$, one has
$$\rk(G,w)\leq \max\{\rk(G-i,w)+1,\rk(G\ominus i,w)+1,3\}.$$
\end{lemma}

\begin{proof}
Recall the polynomial $f_{G,w}=\alpha(G,w)-p_{G,w}$ from (\ref{fGw}).
 For convenience we consider the node $i=n$ and we set
$\underline{x}=(x_1,x_2,\dots,x_{n-1})$ so that   $x=(\underline{x},x_n)$.
By Lemma \ref{lemp2},
\begin{equation}\label{relation1}
f_{G,w}(x)=(1-x_n)f_{G,w}(\underline{x},0)+x_n f_{G,w}(\underline{x},1).
\end{equation}
First,  we can write $f_{G,w}(\underline{x},0)=f_{G-n,w}(\underline{x}) +g_1$, where
$g_1=\alpha(G,w)-\alpha(G-n,w)\ge 0$.
Moreover, we have the identity
$f_{G,w}(\underline{x},1)=f_{G\ominus n,w}(\underline{x})+g_2(\underline{x})$, after setting
$$g_2(\underline{x})=\underbrace{\alpha(G,w)-w_n-\alpha(G\ominus n,w)}_{\ge 0}+\sum_{i\in
N(n)}\underbrace{(w_{in}-w_i)}_{\ge 0}x_i+
\sum_{ij\in {E(G-n)\backslash E(G\ominus n)}}\underbrace{w_{ij}}_{\ge 0} x_ix_j\in H_2.$$
Here we have used the assumption (\ref{wedge1}) in order to claim that $w_{in}\ge w_i$ for all $i\in N(n)$.
Combining with (\ref{relation1}), we obtain
\begin{equation*}
f_{G,w}(x)=(1-x_n)f_{G-n,w}(\underline{x})+ x_nf_{G\ominus n,w}(\underline{x})+ h(x), 
\end{equation*}
where $h(x)=  (1-x_n)g_1+x_ng_2(\underline{x}) \in H_3$. Hence the lemma is proved.
\qed\end{proof}

\begin{proof}{\em (of Theorem \ref{thmupperbound})}
We show (\ref{upperboundrk})  by induction on the
number of nodes $|V(G)|$. If $G$ has no edge then $\rk(G,w)=1$ and thus  the result holds for
$|V(G)|=1$. If $\alpha(G)=|V|-1$ then $G$ is bipartite and thus $\rk(G,w)=2$ (by Corollary \ref{corbip}) and thus the result holds.
Assume now that $|V(G)|\geq 2$ and $\alpha(G)\le |V(G)|-2$. Then there exists  a node $i\in V$ satisfying
\begin{equation*}\label{relationnode}
\alpha(G-i)=\alpha(G).
\end{equation*}
In particular, $i$ is adjacent to at least one node: $|N(i)|\ge 1$.
Using the
 induction assumption for the graphs $G-i$ and $G\ominus i$, we obtain that
 $$\rk(G-i,w)\le (|V(G)|-1)-\alpha(G-i)+1=|V(G)|-\alpha(G-i)=|V(G)| -\alpha(G),$$
 $$\rk(G\ominus i,w) \le( |V(G)| -|N(i)| -1)-\alpha(G\ominus i)+1 =
 |V(G)|-|N(i)|-\alpha(G\ominus i) \le |V(G)| -\alpha(G).$$
 Here we have used the (easy to check) inequality $\alpha(G)\le \alpha(G\ominus i) + |N(i)|$.
Now we can use   Lemma~\ref{thmvup} and conclude that  $\rk(G,w)\le
|V(G)|-\alpha(G)+1$.
\qed\end{proof}

\subsubsection{The second upper bound}

We now give another upper bound for the Handelman rank of a
weighted graph $(G,w)$, which depends on the specific node
weights. Consider an inequality  $w^Tx\le b$ which is valid for $\ST(G)$, where we assume $w\in\oN^V$ and $b\in\oN$; obviously $b\ge \alpha(G,w)$.
Define the {\em defect} of this inequality as
\begin{equation}\label{defect}
\DEF_G(w,b)=2(\alpha^*(G,w)-\min\{b,\alpha^*(G,w)\}).
\end{equation}
Note that the defect is a nonnegative integer number, since the
node weights $w$ are integer valued and there is a
$\{0,1/2,1\}$-valued vector $x\in \FR(G)$ maximizing $w^Tx$
over $\FR(G)$ (see \cite[Section~2.c]{NT74}). We have the following result on the polynomial
$b - p_{G,w}$.

\begin{theorem}\label{thm 2nd upper bound}
Assume $w^Tx \le b$ is valid for $\ST(G)$, where $w\in \oN^V$ and $b\in \oN$, and let the edge
weights satisfy (\ref{wedge1}). Then the polynomial $b -
p_{G,w}$ belongs to $H_{r+2}$, where $r=\DEF_G(w,b)$ is defined in (\ref{defect}).
\end{theorem}

The proof uses the result of Lov\'asz and Schrijver \cite{LS91}
from Lemma \ref{lemls} below. It is along the similar lines as
their proof of  \cite[Theorem 2.13]{LS91} where they upper
bound the $N$-index of the inequality $w^Tx\le \alpha(G,w)$ by
the quantity $2(\alpha^*(G,w)-\alpha(G,w))$. We return
to  the construction of Lov\'asz and Schrijver
\cite{LS91} in Section \ref{lshierarchy}.

\begin{lemma}\label{lemls}\cite[Lemma 2.12]{LS91}
Consider node weights $w\in \oN^V$ for which
$$\alpha(G,w)<\alpha^*(G,w).$$
Then, there exists a node  $i\in V$ such that every vector
$x\in\FR(G)$ maximizing $w^Tx$ over $\FR(G)$ (i.e.,
$w^Tx=\alpha^*(G,w)$) satisfies  $x_i={1\over 2}$.
\end{lemma}

\begin{proof} {\em (of Theorem \ref{thm 2nd upper bound})}
The proof is by induction on the defect
$r:=2(\alpha^*(G,w)-\min\{b,\alpha^*(G,w)\})$. If $r=0$, i.e.,
$b\ge\alpha^*(G,w)=\rho_2(G,w)$, then the result follows from  Proposition \ref{lem1}, since
$b-p_{G,w}= (b-\rho_2(G,w))+(\rho_2(G,w) -p_{G,w})\in H_2$.

Assume now that $b<\alpha^*(G,w)$ (i.e., $r>0$). Then $\alpha(G,w)\le b<\alpha^*(G,w)$ and thus Lemma
\ref{lemls} can be applied. Hence  there exists one node, denoted as $n$ for
convenience, such that every vector $x\in \FR(G)$ optimizing
$w^Tx$ over $\FR(G)$ has $x_n=1/2$. This trivially implies
$w_n>0$.
Let $w_{G-n}$ denote the restriction of $w$ to the nodeset of $G-n$ and define $w'\in \oR^V$ which coincides with $w$ except $w'_n=0$.
Analogously, $w_{G\ominus n}$ denotes the restriction of $w$ to the nodeset of $G\ominus n$ and $w''\in \oR^V$ coincides with $w$ except $w''_i=0$ if $i$ is equal or adjacent to $n$.
Observe that $\alpha^*(G,w')=\alpha^*(G-n,w_{G-n})$ and $\alpha^*(G,w'')=\alpha^*(G\ominus n, w_{G\ominus n})$.

We consider the two inequalities $w_{G-n}^Tx\le b$ and
$w_{G\ominus n}^Tx\le b-w_n$, which are clearly valid for
$\ST(G-n)$ and $\ST(G\ominus n)$, respectively. Their defects
are respectively denoted as
$r'=2(\alpha^*(G-n,w_{G-n})-\min\{b, \alpha^*(G-n,w_{G-n})\}) =
2(\alpha^*(G,w')- \min\{b,\alpha^*(G,w')\})$ and
$r''=2(\alpha^*(G\ominus n,w_{G-n})-\min\{b-w_n,
\alpha^*(G\ominus n,w_{G-n})\})=
2(\alpha^*(G,w'')-\min\{b-w_n,\alpha^*(G,w'')\})$. We show that
both  defects smaller than $r$, i.e., that $r',r''<r$.

First, we show that $r'<r$.
This is clear if  $b\ge \alpha^*(G,w')$ as  then $r'=0 <r$.
Now, we can suppose that
$b< \alpha^*(G,w')$ and it suffices to show that $\alpha^*(G,w')<\alpha^*(G,w)$.
For this, let $y$ be a vertex of $\FR(G)$
maximizing $(w')^Tx$ over $\FR(G)$. Then,
$$w^Ty=(w')^Ty+w_ny_n=\alpha^*(G,w')+w_ny_n\le\alpha^*(G,w).$$
 If
$y_n>0$, then $\alpha^*(G,w')\le \alpha^*(G,w)-w_ny_n <\alpha^*(G,w)$, since
 $w_n>0$. If $y_n=0$ then, by Lemma \ref{lemls}, $y$ does not maximize $w^Tx$ over $\ST(G)$ and thus $w^Ty<\alpha^*(G,w)$, giving again
 $\alpha^*(G,w')<\alpha^*(G,w)$. Thus $r'<r$ holds.

We now show that $r''<r$. This is clear if
 $b-w_n\ge \alpha^*(G,w'')$ as then $r''=0<r$.  Now,
we can suppose that $b-w_n< \alpha^*(G,w'')$ and it suffices to show that $\alpha^*(G,w'')+w_n<\alpha^*(G,w)$.
For this let $z$ be a
vertex of $\FR(G)$ maximizing $(w'')^Tx$ over $\FR(G)$. Define the new vector $\bar{z}\in \oR^V$ which coincides with $z$ except $\bar{z}_n=1$ and $\bar{z}_i=0$ if $i$ is adjacent to $n$.
Then, $\bar{z}\in \FR(G)$ and
$w^T\bar{z}=(w'')^Tz+w_n=\alpha^*(G,w'')+w_n$.
As $\bar{z}_n\ne \frac{1}{2}$, we deduce from  Lemma \ref{lemls} that
$w^T\bar{z}<\alpha^*(G,w)$ thus showing $\alpha^*(G,w'')+w_n<\alpha^*(G,w)$.

Thus $r'+2, r''+2\le r+1$ and using the  induction assumption we can conclude that the following two polynomials both lie in the
Handelman set of order $r+1$:
\begin{eqnarray*}
f_1 &=& b-\sum_{i\in V(G-n)}w_ix_i+\sum_{ij\in E(G-n)}w_{ij}x_ix_j\in H_{r+1},\\
f_2 &=& b-w_n-\sum_{i\in V(G\ominus n)}w_ix_i+\sum_{ij\in E(G\ominus n)}w_{ij}x_ix_j\in H_{r+1}.
\end{eqnarray*}
Define $f:=b-p_{G,w}$ and  observe that
$$f(\underline{x},0)=f_1\ \
\text{and}\ \
f(\underline{x},1)=f_2+\sum_{i\in N(n)}(w_{in}-w_i)x_i+\sum_{ij\in { E(G-n)
\backslash E(G\ominus n)}}w_{ij}x_ix_j.$$
\noindent By Lemma \ref{lemp2},  $f(x)= (1-x_n) f(\underline{x},0) +x_n f(\underline{x},1),$ thus implying
$f\in H_{r+2}$.
\qed\end{proof}

Considering that the defect of $w^Tx\le \alpha(G,w)$ is
$2(\alpha^*(G,w)-\alpha(G,w))$, by Theorem \ref{thm 2nd upper
bound} we have the following upper bound for $\rk(G,w)$.

\begin{corollary}\label{2nd upper bound}
Consider a weighted graph $(G,w)$ with integer node weights
$w\in\oN^V$ and where the edge weights satisfy (\ref{wedge1}).
Then,
\begin{equation}\label{2upper bound}
\rk(G,w)\le 2(\alpha^*(G,w)-\alpha(G,w)) +2.
\end{equation}
\end{corollary}

\begin{remark}
The upper bound (\ref{upperboundrk}) holds for any weight
function $w\in\oR^V_+$, while the upper bound (\ref{2upper
bound}) holds  for integral weight function $w\in\oN^V$ {(which can be assumed without loss of  generality)}.
It turns out that  these two upper bounds  are
not comparable. Indeed, for the unweighted odd circuit $C_{2n+1}$,
(\ref{upperboundrk}) and (\ref{2upper bound}) give $n+2$ and
$3$, respectively. On the other hand, consider an unweighted
graph consisting of $n$ isolated nodes, then
(\ref{upperboundrk}) and (\ref{2upper bound}) read $1$ and $2$,
respectively.
\end{remark}

\subsection{Handelman ranks of some special classes of graphs}\label{sec33}

As an application we can now determine the Handelman rank of
some special classes of graphs, including perfect graphs, odd
circuits and their complements.

\subsubsection{Perfect graphs}

A graph $G$ is said to be {\em perfect} if equality $\omega(H)=\chi(H)$ holds
for all induced subgraphs $H$ of $G$ (including $H=G$).
We will use the following properties of perfect graphs and refer  to \cite{Lov72} for details. If $G$ is perfect then its complement $\overline G$ is perfect as well and thus $\alpha(H)=\chi(\overline H)$ for all induced subgraphs $H$ of $G$.
Moreover, $\alpha(G,w)=\chi(\overline G,w)$ for any  node weights $w\in \oR^V_+$.
We also use the following well-known fact: For any graph $G$, $|V(G)|\le \alpha(G) \chi(G),$ with equality if $G$ is perfect and vertex transitive
(see e.g. \cite[Section 67.4]{Schrijver}).
We can show the following upper bound for the Handelman rank of weighted perfect graphs.

\begin{proposition}\label{propperfect}
 Consider a  weighted graph $(G,w)$ where the edge weights satisfy (\ref{wedge1}). If  $G$ is perfect then $\rk (G,w)\le \omega(G)$. Moreover,  in the unweighted case, $\rk(G)=\omega(G)$   if  $G$ is vertex-transitive.
 \end{proposition}

 \begin{proof}
We know from Proposition \ref{lem1} that $\chi(\overline{G},w) -p_{G,w}\in H_{\omega(G)}$.
 As $G$ is perfect,   $\alpha(G,w)=\chi(\overline G,w)$ and thus $\alpha(G,w)-p_{G,w}\in H_{\omega(G)}$, which shows $\rk(G,w)\le \omega(G)$.
 Assume now that $w$ is the all-ones vector and that $G$ is perfect and vertex-transitive.
 Then, we have equality: $|V(G)|=\alpha(G)\chi(G) =\alpha(G)\omega(G)$. Using Proposition \ref{lembound2}, we obtain that
  $\rk(G)\ge |V(G)|/\alpha(G)= \omega(G)$, which implies $\rk(G)=\omega(G)$.
 \qed\end{proof}

\begin{figure}
\centerline{\includegraphics[width=0.25\textwidth]{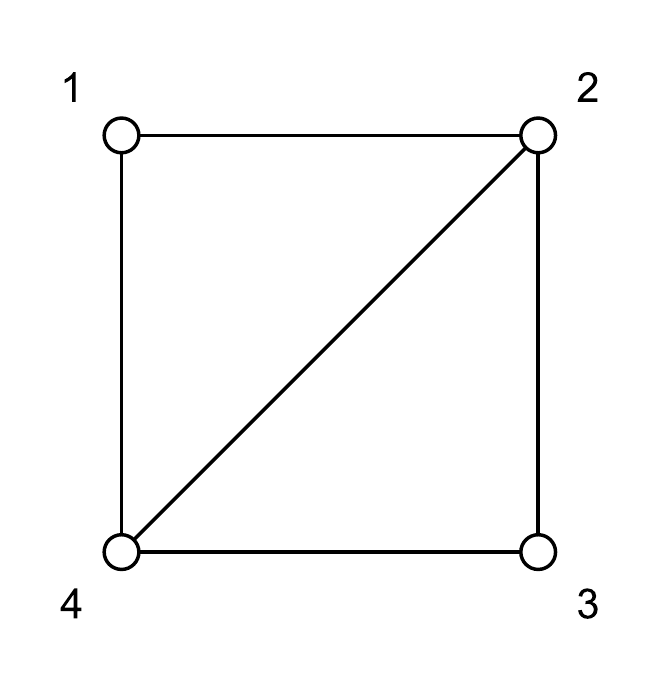}}
\caption{}\label{fig1}
\end{figure}

\begin{remark}\label{remex}
The inequality $\rk(G)\le
\omega(G)$ can be strict for some perfect graphs.
This is the case, for instance, for the graph $G$  from
Example \ref{exsmallrkH}, which is perfect
with $\omega(G)=t+1$ and $\rk(G)= 2$.
 Figure \ref{fig1} shows this graph for the case $t=2$.
\end{remark}

\subsubsection{Odd circuits and their complements}

Park and Hong \cite{PH12} show that
the Handelman rank of an odd circuit is equal to 3.
Here we  show  that  the Handelman
rank of   a weighted odd circuit is at most 3, answering an open question of \cite{PH12}, and we also consider  the Handelman rank of complements of odd circuits.

\begin{proposition}\label{lemoddcircuit}
Consider a  weighted odd circuit $(C_{2n+1},w)$
 and its complement $(\overline{C_{2n+1}},w)$, where the edge weights satisfy (\ref{wedge1}).
Then,
$$\rk(C_{2n+1},w)\leq 3 \ \text{ and } \
\rk(\overline{C_{2n+1}},w)\leq n+1.$$ Moreover, equality holds
in the unweighted case: $\rk(C_{2n+1})= 3$ and
$\rk(\overline{C_{2n+1}})= n+1$.
\end{proposition}

\begin{proof}
For any node $i$,  both graphs $C_{2n+1}-i$ and
${C_{2n+1}}\ominus i$ are bipartite  and thus $\rk(C_{2n+1}-i,w),$ $\rk (C_{2n+1}\ominus i,w)\le 2$
 by Corollary \ref{corbip}.
 Applying Lemma \ref{thmvup}, we obtain that $\rk(C_{2n+1},w)\le 3$.
Similarly, for any node $i$,
both graphs $\overline{C_{2n+1}}-i$ and $\overline{C_{2n+1}}\ominus i$
are perfect with clique number at most $n$ and thus,  from  Proposition \ref{propperfect},
 $\rk(\overline{C_{2n+1}}-i,w)$, $\rk(\overline{C_{2n+1}}\ominus i,w)\le n$.
Applying again Lemma~\ref{thmvup} we deduce that
$\rk(\overline{C_{2n+1}},w)\le n+1$. In the unweighted case,
the lower bounds
$\rk(C_{2n+1})\ge 3$ and $\rk(\overline{C_{2n+1}})\ge  n+1$ follow from Proposition \ref{lembound2}. Indeed,  
 $\rk(C_{2n+1})\ge {2n+1\over \alpha(C_{2n+1})}={2n+1\over n}>2$
 and
 $\rk(\overline{C_{2n+1}})\ge {2n+1\over
\alpha(\overline{C_{2n+1}})}={2n+1\over 2}>n$.
\qed\end{proof}

As an application we obtain  the
following characterization of perfect graphs, which is in the same spirit as the following
well-known characterization due to Lov\'asz \cite{Lov72}:
$G$ is perfect if and only if  $|V(H)|\le
\alpha(H)\omega(H)$ for all induced subgraphs $H$ of $G$.

\begin{corollary}
\label{corperfectchar} A graph $G$ is perfect if and only if
$\rk(H)\le \omega(H)$ for every induced subgraph $H$ of $G$.
\end{corollary}

\begin{proof}
The `only if' part follows from Proposition \ref{propperfect}.
Conversely, assume that $G$ is not perfect. Using the
perfect graph theorem of Chudnovsky, Robertson, Seymour and
Thomas \cite{CRST06}, we know that $G$ contains an induced subgraph $H$
which is an odd circuit or its complement. By Proposition \ref{lemoddcircuit},
$\rk(H)
=\chi(H)>\omega(H)$,  concluding the proof. \qed\end{proof}

\begin{remark}\label{constructive proof}
As noted earlier,  the upper bound 3 for the Handelman rank of an odd circuit  also follows from the upper bound from Corollary \ref{2nd upper bound} in terms of the defect.
Indeed, $\alpha^*(C_{2n+1})= (2n+1)/2$,  so that the defect of the inequality $\sum_{i\in V(C_{2n+1})}x_i \le n=\alpha(C_{2n+1})$ is equal to
$2((2n+1)/2 -n)= 1$ and thus relation  (\ref{2upper bound}) gives the upper bound 3.

Park and Hong
\cite{PH12} show that the Handelman rank of an odd circuit  is at most 3 by constructing an explicit decomposition of the polynomial
$\alpha(C_{2n+1})-p_{C_{2n+1}}$ in the Handelman set $H_3$.
We illustrate their argument for the case of $C_5$, see Figure \ref{figC5}.
Then,   we have:
$$\alpha(C_5)-p_{C_5} = 2-\sum_{i=1}^5x_i+ \sum_{i=1}^4 x_ix_{i+1}+x_1x_5= f_{123} +f_{145} + f'_{1,34},$$
where
$$f_{123}=1-(x_1+x_2+x_3) +x_1x_2+x_1x_3+x_2x_3=(1-x_1)(1-x_2)(1-x_3)+x_1x_2x_3\in
H_3,$$
$$f_{145}=1-(x_1+x_4+x_5)+x_1x_4+x_1x_5+x_4x_5=(1-x_1)(1-x_4)(1-x_5)+x_1x_4x_5\in
H_3,$$
$$f'_{1,34} = f_{134}(1-x_1,x_2,x_3)=x_1-x_1x_3-x_1x_4+x_3x_4=
x_1(1-x_3)(1-x_4)+ (1-x_1)x_3x_4 \in H_3.$$
In the above decomposition,  $f_{123}$ and $f_{145}$ are the polynomials corresponding to the two cliques
$\{1,2,3\}$ and $\{1,4,5\}$  (obtained by adding the
edges 13 and 14 to $C_5$), and  the polynomial $f'_{1,34}$ permits to cancel the quadratic terms $x_1x_3$ and $x_1x_4$ corresponding to the added   edges 13 and  14 and to add the quadratic term $x_3x_4$.
This construction extends easily  to an arbitrary odd circuit, showing
$\rk(C_{2n+1})\le 3$.
\end{remark}

\begin{figure}
\centerline{\includegraphics[width=0.4\textwidth]{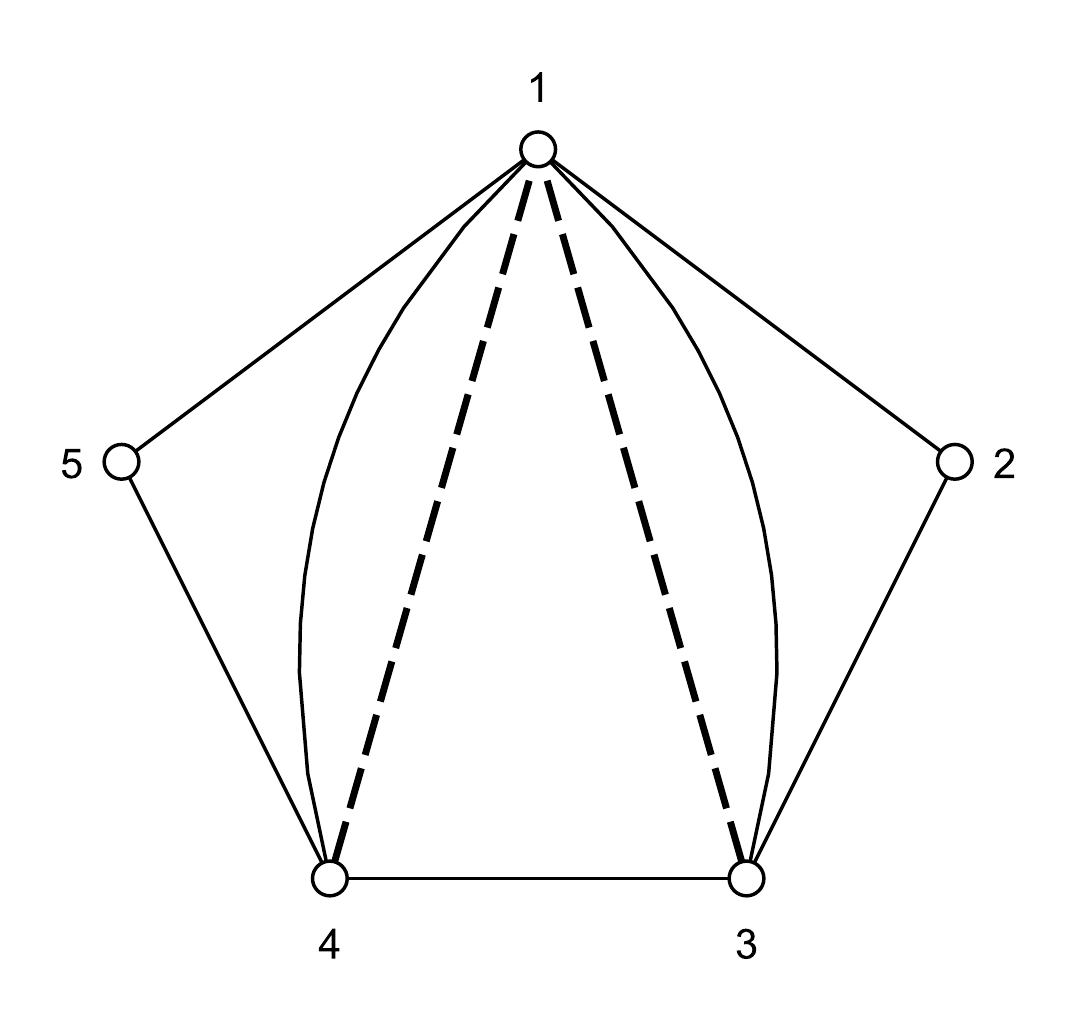}}
\caption{Odd circuit $C_5$}\label{figC5}\
\end{figure}

We conclude with bounding the Handelman rank of two more classes of graphs.

\begin{example}\label{exwheel}
Consider the odd wheel $W_{2n+1}$, which is the graph obtained from an odd circuit $C_{2n+1}$   by adding a new node (the apex node, denoted as $v_0$) and making it adjacent to all nodes of $C_{2n+1}$.
Since by deleting the apex node $v_0$ one obtains $C_{2n+1}$ with Handelman rank 3, Lemma \ref{thmvup} implies that the Handelman rank of the wheel $W_{2n+1}$ is at most 4; note that this bound also holds for any weighted wheel.
Moreover,  the complement of $W_{2n+1}$ has the same Handelman rank as the complement of $C_{2n+1}$ (since node $v_0$ is isolated, and apply Lemma \ref{theodnode} (iv) below).
\end{example}
\begin{figure}
\centerline{\includegraphics[width=0.7\textwidth]{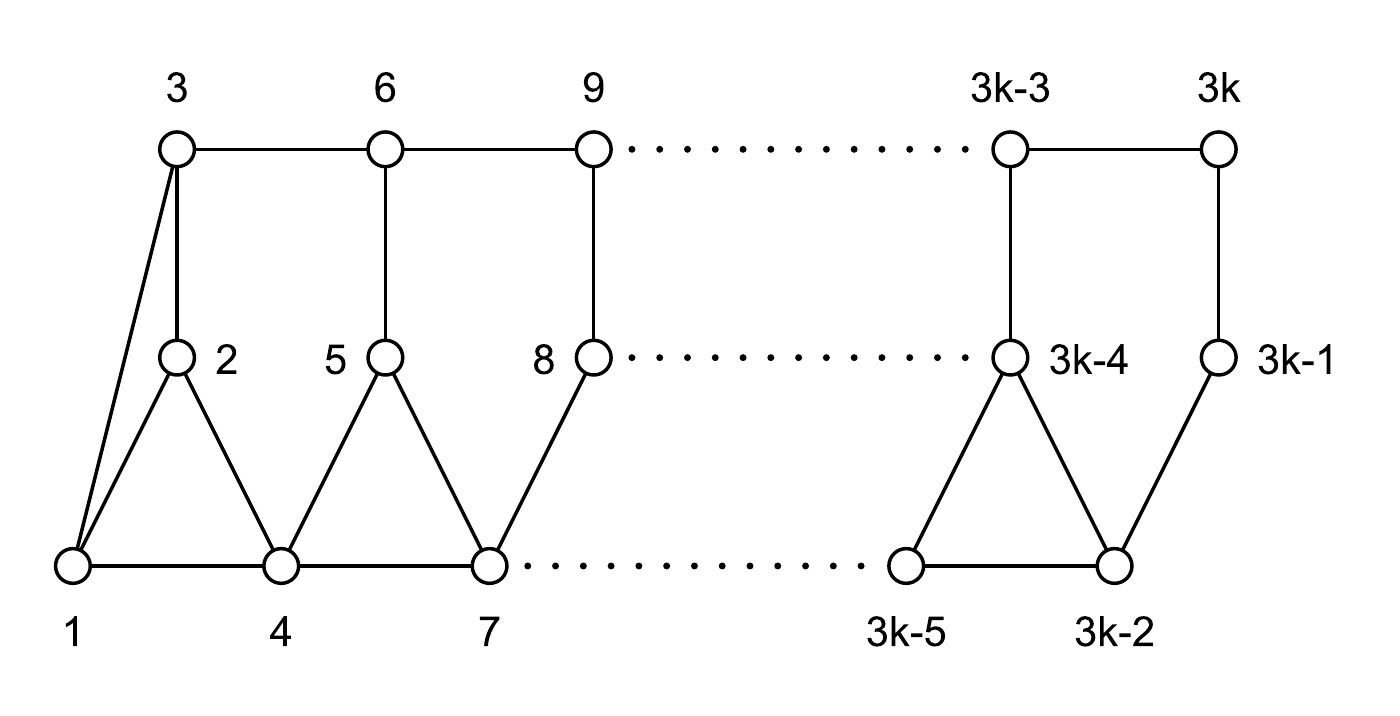}}
\caption{Graph $G_k$}\label{fig3} 
\end{figure}
\begin{example}\label{exLT}
We now consider the graphs $G_k$, constructed by Lipt\'ak and Tuncel
\cite{LT03} and
 defined as in
Figure \ref{fig3}.
Hence, for $k=2$, $G_2$ is the circuit $C_5$ with a new node adjacent to three consecutive nodes of $C_5$.
We  show that,  for any $k\ge 2$,
the Handelman rank of the graph $G_k$ is equal to $3$ or $4$.

As $G_k$ has $3k$ nodes and $\alpha(G_k)= k$, the lower bound (\ref{lowerboundrk})
for the Handelman rank gives  $\rk(G_k)\ge 3$.
 Now, we look at the upper bound
for the Handelman rank. First, we consider the case $k=2$. As
in Remark \ref{constructive proof}, we can give an explicit
decomposition for the polynomial $\alpha(G_2)-p_{G_2}$, obtained by adding the chords $(3,4)$ and $(4,6)$ to $G_2$. Namely,
$$\alpha(G_2)-p_{G_2}=f_{1234}+f_{456}+f'_{4,36},$$
where
$$f_{1234}=1-\sum_{i=1}^4x_i+\sum_{1\le i<j\le 4}x_ix_j\in H_4,$$
$$f_{456}=1-\sum_{i=4}^6x_i+(x_4x_5+x_4x_6+x_5x_6)\in H_3,$$
$$f'_{4,36}=f_{436}(1-x_4,x_3,x_6)=x_4(1-x_3)(1-x_6)+(1-x_4)x_3x_6\in H_3.$$
In the above decomposition,  $f_{1234}$ and $f_{456}$ are the
polynomials corresponding to the two cliques $\{1,2,3,4\}$ and
$\{4,5,6\}$  (obtained by adding the edges $34$ and $46$ to $G_2$),
and  the polynomial $f'_{4,36}$ permits to cancel the quadratic
terms $x_3x_4$ and $x_4x_6$ corresponding to the added edges $34$
and $46$ and to add the quadratic term $x_3x_6$.

This construction extends easily to an arbitrary $k\ge 3$,
showing $\rk(G_{k})\le 4$. For example,
$\alpha(G_3)-p_{G_3}=f_{1234}+f_{4567}+f_{789}+f'_{4,36}+f'_{7,69}\in
H_4$.

Observe that the upper bound from Corollary \ref{2nd upper bound} is  not strong enough to show this.
Indeed the defect of the inequality $\sum_{i\in V(G_k)}x_i \le \alpha(G_k)=k$ is equal to
 $2(\alpha^*(G_k)-\alpha(G_k))=k$, since  $\alpha(G_k)=k$ and $\alpha^*(G_k)= 3k/2$
(this follows from the fact $\sum_{i\in V(G_k)}x_i\le \alpha(G_k)$ defines a facet of $\ST(G_k)$, shown in \cite[Lemma 32 and Theorem 34]{LT03}, so that $\alpha^*(G_k)=3k/2$ by Lemma 2.10 of \cite{LS91}). Thus Corollary \ref{2nd upper bound} permits only to conclude that $\rk(G_k)\le k+2$.
\end{example}

\subsection{Graph operations}\label{sec34}

In this subsection, we investigate  the behavior of the
Handelman rank under some graph operations like node or edge
deletion, edge contraction, and taking clique sums. For
simplicity, we only consider unweighted graphs, while some of
the results can easily be extended to the weighted case.

\subsubsection{Operations on edges and nodes}
An interesting observation is that the Handelman rank is not
monotone under edge deletion. As an illustration, look at the
three graphs in Figure \ref{fig4}. Consider the first complete
graph $K_4$ with $\rk(K_4)=4$. If we delete one edge (say edge
13), we obtain the second graph $G$ with rank $\rk(G)=2$.
However, if we additionally delete the edges 12 and 14, then
the third graph $G'=K_4\setminus \{12,13,14\}$ has $\rk(G')=3$,
since it is the clique 0-sum of a node and a clique of size 3.
(See Lemma \ref{lemcliquesum} below.)
On the other hand,  if we delete an edge whose deletion
increases the stability number (a so-called  {\em critical edge}), then
the Handelman rank does not increase.

\begin{figure}
\centerline{\includegraphics[width=0.9\textwidth]{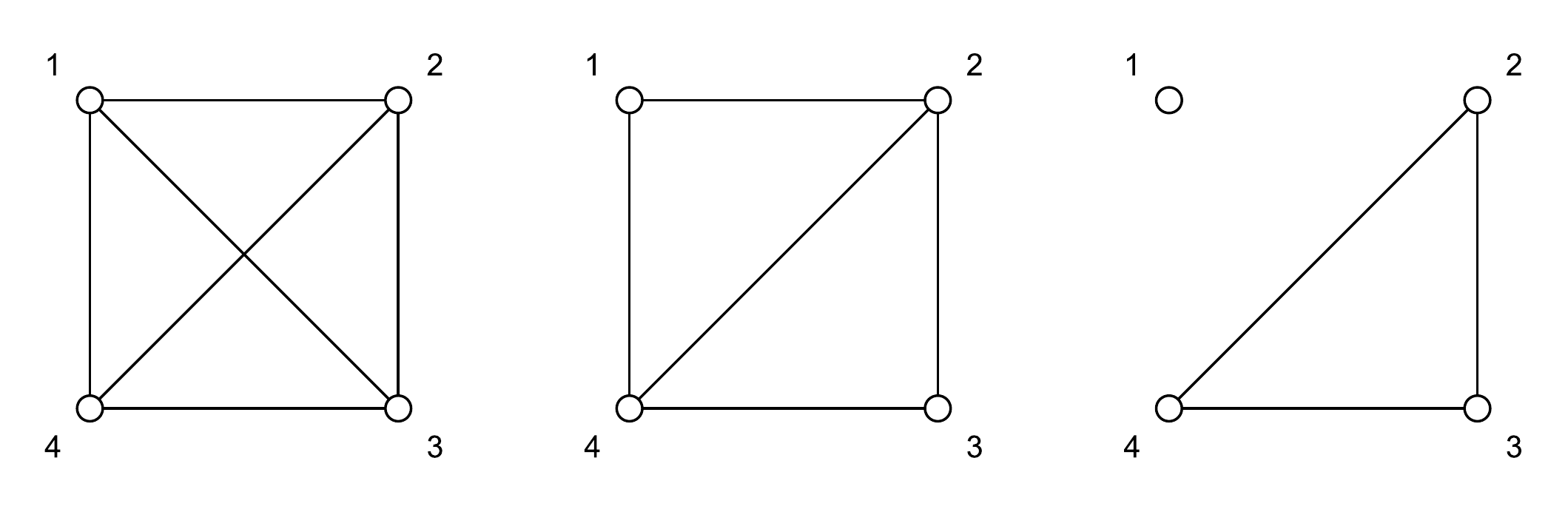}}
\caption{}\label{fig4}
\end{figure}

\begin{lemma}\label{leme}
Let $e$ be an edge of $G$ such that $\alpha (G\setminus
e)=\alpha(G)+1$. Then,
$\rk(G\setminus e)\le \rk(G)$.
\end{lemma}

\begin{proof}
Say $e$ is the edge 12. Then,
    $\alpha(G\setminus e)-p_{G\setminus e}=\alpha(G)-p_G+1 -x_1x_2.$ As
$1-x_1x_2 = 1-x_2+x_2(1-x_1)\in H_2$, this implies that
$\rk(G\setminus e)\le \rk(G)$.
\qed\end{proof}

The Handelman rank is not monotone under edge contraction
either. For instance, the graph $G$ in Figure \ref{fig1} has
$\rk(G)=2$. If we contract the edge 23, we get the new graph
$G'$ is a triangle with $\rk(G')=3$. If we contract one more
edge 12, the resulting graph $G''$ is an edge with
$\rk(G'')=2$. Analogously, deleting a node can either increase,
decrease or not affect the Handelman rank. We group several
properties about the behavior of the Handelman rank under node
deletion.

\begin{lemma}\label{theodnode}
Let $G=(V,E)$ be a graph and $j\in V$.
\begin{itemize}
\item[(i)] If  $\alpha (G-j)=\alpha (G)$, then $\rk(G-j)\le
    \rk(G).$
       \item[(ii)] If $\alpha (G-j)=\alpha(G)-1$, then
           $\rk(G)\le   \rk(G-j)$.
    \item[(iii)] If $j$ is adjacent to all other nodes of
        $G$, then $\rk(G)\le \rk(G-j)+1$.
\item[(iv)] If $j$ is an isolated node, then
    $\rk(G)=\rk(G-j)$.
\end{itemize}
\end{lemma}

\begin{proof}
(i) We use relation (\ref{relation1}) applied to the polynomial
$f_G$ (and node $j$). As before $\ux$ consists of all variables
except $x_j$, so that $x=(\ux,x_j)$. As
$\alpha(G-j)=\alpha(G)$, we have  $f_{G-j}(\ux) =f_G(\ux,0)\in
H_{\rk(G)}$, which implies $\rk(G-j)\le \rk(G)$.

(ii) If $\alpha(G-j)=\alpha(G)-1$, then
$f_G=f_{G-j}+(1-x_j)+\sum_{i: ij\in E}w_{ij}x_ix_j\in
H_{\rk(G-j)}$. Hence, $\rk(G)\leq \rk(G-j)$.

(iii) Assume that  $j$ is adjacent to all other nodes of $G$.
If $G-j$ has no edge then $G$ is bipartite and thus
$\rk(G)=2=\rk(G-j)+1$. Assume now that $G-j$
has an edge so that $\rk(G-j)\ge 2$.  Using Lemma
\ref{thmvup}, we deduce that $\rk(G)\le \rk(G-j)+1$.

(iv) $G$ is  the clique 0-sum of $G-j$ and the single node
$j$, and we can apply Lemma \ref{lemcliquesum} below.
\qed\end{proof}

\begin{remark}
In Lemma \ref{theodnode} (ii), the gap $\rk(G-j)-\rk(G)$ can be
arbitrarily large. To see this consider the graph $G$  obtained
by taking the clique $t$-sum of $K_{2t}$ and $K_{t+1}$  along a
common $K_t$. Let $j$ be the node of $K_{t+1}$ which does not
belong to the common clique $K_t$. If we delete node $j$, then
$G-j=K_{2t}$ has $\rk(G-j)=2t$. On the other hand, $\rk(G)\le
t+1$, since $\alpha (G)=2=\rho_{t+1}(G)$ as $V(G)$ can be
covered by two cliques of size at most $t+1$. Thus
$\rk(G-j)-\rk(G)\ge 2t-(t+1)=t-1$.
\end{remark}

\subsubsection{Clique sums}

Suppose $G=(V,E)$ is the clique $t$-sum of two graphs  $G_1$ and $G_2$.
We now study the Handelman rank of $G$,  whose value needs technical case checking, depending on the values of the stability numbers of $G$, $G_1$, $G_2$ and of some subgraphs.

\begin{lemma} \label{lemcliquesum}
Suppose $G$ is the clique $t$-sum of $G_1$ and $G_2$ along a common $t$-clique
 $C_0$ and let $H_i=G_i\setminus C_0$ for $i=1,2$.
 The following holds.
\begin{itemize}
\item[(i)] If $\alpha(G)=\alpha(G_1)+\alpha(G_2)$, then
    $$\rk(G)\leq \min\{\max\{\rk(G_1),\rk(H_2)\}, \max\{\rk(H_1), \rk(G_2))\}\}.$$
 Moreover, $\rk(G)\le \max\{\rk(G_1),\rk(G_2)\}$ if  $t\le 3$.
\item[(ii)]  Assume  $\alpha(G)=\alpha(G_1)+\alpha(G_2)-1$. Then  $\alpha(G_k)=\alpha(H_k)+1$ for (say) $k=1$ and
    $\rk(G)\leq \max\{\rk(H_1),\rk(G_2)\}$.
\item[(iii)] Assume  $\alpha(G)=\alpha(G_1)+\alpha(G_2)-2$. For $k\in \{1,2\}$ let $C_k$ denote the set of nodes of $C_0$ which belong to at least one maximum stable set of $G_k$.
Set  $H'_1=G_1\backslash C_1$ and $H'_2=G\setminus H_1'=G_2\setminus (C_0\setminus C_1)$. Then $\alpha(H'_k)=\alpha(G_k)-1$  for $k=1,2$,
  and $\rk(G)\leq   \max\{\rk(H'_1),\rk(H'_2)\}$.
\end{itemize}
\end{lemma}

\begin{proof}In what follows, for subsets $A,B\subseteq V$,
 $E(A,B)$  denotes the set of edges $ij$ with $i\in A$ and $j\in B$, and $E(A)$ the set of edges contained in $A$. We also set $VG$ for $V(G)$.\\
(i) We use the identities
\begin{equation*}
f_G=f_{G_1}+f_{H_2}+(\alpha(G)-\alpha(G_1)-\alpha(H_2))+\sum_{ij \in E(VG_1,VH_2)}x_ix_j,
\end{equation*}
\begin{equation*}
f_G=f_{G_2}+f_{H_1}+(\alpha(G)-\alpha(G_2)-\alpha(H_1))+\sum_{ij \in E(VG_2, VH_1)}x_ix_j.
\end{equation*}
As $\alpha(G)=\alpha(G_1)+\alpha(G_2)$,
$\alpha(G)-\alpha(G_1)=\alpha(G_2)\geq \alpha(H_2)$ 
and $\alpha(G)-\alpha(G_2)=\alpha(G_1)\geq \alpha(H_1)$, implying
 $\rk(G)\leq \min\{\max\{\rk{(G_1)},\rk(H_2)\},\max\{\rk{(G_2)},\rk{(H_1)}\}\}.$
For the second statement, we use the identity
$$f_G=f_{G_1}+f_{G_2}+\sum_{i\in C_0}x_i-\sum_{ij\in E(C_0)}x_ix_j$$
combined with the fact that  $\sum_{i\in C_0}x_i-\sum_{ij \in E(C_0)}
x_ix_j \in H_2$ when $t=|C_0|\le 3$. This is clear if $t\le 1$ and follows from the identities
$x_1+x_2-x_1x_2 =x_1(1-x_2)+x_2\in H_2$ and
$x_1+x_2+x_3-x_1x_2-x_1x_3-x_2x_3 = x_1(1-x_2)+x_2(1-x_3)+x_3(1-x_1)\in H_2$ if $t=2,3$.
From this follows that  $\rk(G)\le \max\{\rk(G_1),\rk(G_2)\}$.

\smallskip
(ii) As $\alpha(G)\neq \alpha(G_1)+\alpha(G_2)$, it follows that $\alpha(H_k)=\alpha(G_k)-1$ for at least one index $k=1,2$. Say this holds for $k=1$.
Then we use the identities
$$f_G=f_{G_1}+f_{G_2}-1+\sum_{i\in C_0}x_i-\sum_{ij\in E(C_0)}x_ix_j,$$
and $$f_{H_1}=f_{G_1}-1+\sum_{i\in C_0}x_i-\sum_{ij\in E(C_0)}x_ix_j-\sum_{ij\in E(C_0,VG_1\setminus C_0)}x_ix_j.$$
This gives:
$$f_G=f_{H_1}+f_{G_2}+\sum_{ij\in E(C_0,VG_1\setminus C_0)}x_ix_j,$$
which implies $\rk(G)\le \max\{\rk(H_1),\rk(G_2)\}.$

\smallskip
(iii) By construction, $\alpha(H_1')=\alpha(G_1)-1$. Moreover, as $\alpha(G)=\alpha(G_1)+\alpha(G_2)-2$, it follows that $C_1\cap C_2=\emptyset$ and thus
$\alpha(H'_2)=\alpha(G_2)-1$.
We now use the identities
$$f_{H'_1}=f_{G_1}-1+\sum_{i\in C_1} x_i-\sum_{ij\in E(C_1) \cup E(C_1,VG_1\setminus C_1)}x_ix_j,$$
$$f_{H'_2}=f_{G_2}-1+
\sum_{i\in C_0\setminus C_1}x_i-
\sum_{ij\in E(C_0\setminus C_1) \cup E(C_0\setminus C_1, VG_2\setminus (C_0\setminus C_1))} x_ix_j,$$
and $$f_G=f_{G_1}+f_{G_2}-2+\sum_{i\in
C_0}x_i-\sum_{ij\in E(C_0)}x_ix_j.$$
Combining these relations,  we obtain
$$f_G=f_{H'_1}+f_{H'_2}+\sum_{ij\in E(C_1,VG_1\setminus C_1) \cup E(C_0\setminus C_1, VG_2\setminus C_0)} x_ix_j$$
which shows $\rk(G)\le \max\{\rk(H_1'),\rk(H_2')\}$.
\qed\end{proof}

In the special case when  $G$ is a clique sum of two cliques, one can easily determine the
 the exact value of the Handelman rank of $G$.

\begin{lemma} \label{lemchordal}
Assume that  $G$ is the  clique $t$-sum of two cliques $K_{n_1}$ and
$K_{n_2}$ with  $n_1\leq n_2$. Then, $\rk(G)=
\max\{\lceil\frac{n_1+n_2-t}{2}\rceil,n_2-t\}$.
\end{lemma}

\begin{proof}
Obviously, $\alpha(G)=2$. Define $n=|V(G)|=n_1+n_2-t$.
Assume first that  $n_2-n_1\leq t$. Then $V(G)$ can be covered by two cliques
of sizes $\lceil\frac{n}{2}\rceil$ and
$\lfloor\frac{n}{2}\rfloor$ and thus $\rk(G)\le
\lceil\frac{n}{2}\rceil$. In addition, by (\ref{lowerboundrk}),
$\rk(G)\ge {n\over \alpha(G)}={n\over 2}$. Hence we obtain $\rk(G)= \lceil\frac{n}{2}\rceil
= \max\{\lceil\frac{n}{2}\rceil, n_2-t\}$.

Assume now that  $n_2-n_1> t$. Then $G$ can be covered by two cliques of
sizes $n_1$ and $n_2-t$, which implies  $\rk(G)\le n_2-t$. On
the other hand, by applying Lemma \ref{theodnode} (i) to all
nodes $i$ in the common $t$-clique, together with Lemma
\ref{lemcliquesum}, we obtain the reverse inequality
$\rk(G)\geq \max\{\rk(K_{n_2-t}),\rk(K_{n_1-t})\}=n_2-t$.
\qed\end{proof}

\section{Links to other hierarchies}

Several other hierarchies have been considered in the literature for general 0/1 optimization problems applying also to the maximum stable set problem, in particular, by Sherali and Adams \cite{SA90}, by Lov\'asz and Schrijver \cite{LS91}, by Lasserre \cite{Las02}, and by de Klerk and Pasechnik \cite{KP02}.
We briefly indicate how they relate to the Handelman hierarchy considered in this paper, based on optimization on the hypercube.

\subsection{Sherali-Adams and Lasserre hierarchies}

Consider the following $0-1$ polynomial optimization
problem:
\begin{equation}\label{problemsa}
\max \ \ p(x)\ \ {\rm{s.t.}}\ \ x\in K\cap \{0,1\}^n,
\end{equation}
which is obtained by adding the integrality constraint $x\in \{0,1\}^n$ to problem (\ref{problem1}).
Recall that $\II$ denotes the ideal generated by $x_i-x_i^2$ for $i\in [n]$ and that the Handelman set $H_t$ is defined in (\ref{eqsettiH}).
Sherali and Adams \cite{SA90} introduce the following   bounds for (\ref{problemsa}):
\begin{equation}\label{eqsa}
p_{\text{sa}}^{(t)}=\inf\left\{\lambda: \lambda-p \in H_t +\sum_{j=1}^m g_j H_{t-\deg(g_j)} + \II\right\}.
\end{equation}
The above program  is in  fact the dual  of the linear program usually used to define  the Sherali-Adams bounds.
For details we refer e.g. to \cite{SA90,Las02b,Lau03}.


When applying the Sherali-Adams construction  to the maximum stable set problem for the instance $(G,w)$,
the starting point is to formulate  $\alpha(G,w)$ as the problem of maximizing the linear polynomial $p(x)=w^Tx=\sum_{i\in [n]}w_ix_i$
over $K\cap\{0,1\}^n$, where $K=\FR(G)$ is the fractional stable set polytope, so that 
 the corresponding bound from (\ref{eqsa}) reads
\begin{equation}\label{eqsa1}
p_\sa^{(t)}(G,w)=\inf\left\{\lambda: \lambda-w^Tx \in H_t+\sum_{ij\in E} (1-x_i-x_j)H_{t-1} +\II\right\}.
\end{equation}
For $t\ge 2$, let $\langle x_ix_j:ij\in E\rangle _t$ denote the truncated ideal 
consisting of all polynomials $\sum_{ij\in E} u_{ij}x_ix_j$ where $u_{ij}\in \oR[x]$ has degree at most $t-2$.
One can formulate the following variation of the bound (\ref{eqsa1}):
$$\sa^{(t)}(G,w)= \min\{\lambda: \lambda- w^Tx \in H_t +\langle x_ix_j:ij\in E\rangle _t +\II\},$$
which satisfies $\sa^{(t+1)}(G,w)\le p_\sa^{(t)}(G,w)\le \sa^{(t)}(G,w).$
(To see it use, for any edge $ij\in E$, the identities
$1-x_i-x_j = (1-x_i)(1-x_j) - x_ix_j$  and $-x_ix_j= x_i(1-x_i-x_j) + x_i(x_i-1)$.)
Comparing with the hypercube based Handelman bound (\ref{phantgw}), we see that
$$\sa^{(t)}(G,w)\le \phan^{(t)}(G,w),$$
since $\lambda - p_{G,w} =\lambda -w^Tx +\sum_{ij\in E}w_{ij}x_ix_j\in H_t$ implies $\lambda -w^Tx\in H_t +\langle x_ix_j:ij\in E\rangle _t $.

\medskip
We now recall the following semidefinite programming bound of Lasserre \cite{Las02}:
$$\las^{(t)}(G,w)= \min\{t: \lambda-w^Tx\in \Sigma_{2t} +\langle x_ix_j:ij\in E\rangle _t +\II\},$$
where $\Sigma_{2t}$ is the set of polynomials of degree at most $2t$ which can be written as a sum of squares of polynomials.
As is well known, $$\las^{(t)}(G,w)\le \sa^{(t)}(G,w);$$
this can easily be seen by noting that, for any set $T$ with $|T|=t$, we have
$$x^I(1-x)^{T\setminus I}= \underbrace{\prod_{i\in I}x_i^2\prod_{j\in
T\setminus I}(1-x_j)^2}_{\in\Sigma_{2t}} + \underbrace{\left(\prod_{i\in I}x_i\prod_{j\in
T\setminus I}(1-x_j) - \prod_{i\in I}x_i^2\prod_{j\in
T\setminus I}(1-x_j)^2\right)}_{\in\II},$$
where the second term belongs to $\II$ in view of Lemma \ref{lemh2}.
Summarizing, we have
$$\alpha(G,w)\le \las^{(t)}(G,w)\le \sa^{(t)}(G,w)\le \phan^{(t)}(G,w).$$
Hence,  the Sherali-Adams and Lasserre bounds are  at least as strong as the Handelman bound at any given  order $t$, however they are  more expensive to compute. Indeed  the Sherali-Adams bound   is linear but its definition involves more terms, and the Lasserre bound is based on semidefinite programming which is computationally more demanding  than linear programming.
\noindent For more results about the comparison between Sherali-Adams and
Lasserre hierarchies,
see e.g.  \cite{Las02b,Lau03}.

\subsection{Lov\'asz-Schrijver hierarchy}\label{lshierarchy}

Given a polytope $K\subseteq [0,1]^n$, Lov\'asz and Schrijver \cite{LS91} build a hierarchy of polytopes nested between $K$ and
the convex hull of $K\cap \{0,1\}^n$ that finds it after $n$ steps.
When applied  to the maximum stable set problem, one starts with the fractional stable set polytope $K=\FR(G)$. For convenience set
 $\tV =V\cup \{0\}$ (where $0$ is an additional element not belonging to $V$) and define the cone
$$\C(G)=\left\{\lambda{1\choose x}:x\in \FR(G),\lambda\ge 0\right\}\subseteq \oR^{\tV}.$$
Define the following set of symmetric matrices indexed by $\tV$:
$$\M(G)=\{Y\in {\mathcal S}_\tV: Y_{ii}=Y_{0i}\ \forall i\in V,\ Ye_i,Y(e_0-e_i)\in \C(G) \ \forall i\in V\}$$
and the corresponding subset of $\oR^V$:
$$N(\FR(G))=\left\{x\in\oR^V: {1\choose x}=Ye_0\ \ {\text{for some}}\ \ Y\in \M(G)\right\}.$$
For $t\ge 2$, define the $t$-th iterate
$N^t(\FR(G))=N(N^{t-1}(\FR(G)))$, setting $N^1(\FR(G))=N(\FR(G))$. It is shown in \cite{LS91} that
$$\ST(G) \subseteq \ldots \subseteq N^t(\FR(G))\subseteq
N^{t-1}(\FR(G))\subseteq\ldots \subseteq N(\FR(G))\subseteq \FR(G),$$
with equality $\ST(G)=N^n(\FR(G))$.
By maximizing the linear function $w^Tx$ over $N^t(\FR(G))$ we get the bound $\ls^{(t)}(G,w)$ which satisfies
$p_\sa^{(t+1)}(G,w)\le \ls^{(t)}(G,w)$ for $t\ge 1$ (see \cite{LS91,Lau03}).

\newcommand{\rak}{\text{rk}}
\newcommand{\LS}{\text{LS}}

For any $w\in \oR^V_+$, the corresponding inequality $w^Tx\le \alpha(G,w)$ is valid for $\ST(G)$. Following \cite{LS91}, its {\em $N$-index}, denoted as
$\rak_\LS(G,w)$, is the smallest integer $t$ for which the inequality $w^Tx\le \alpha(G,w)$ is valid for $N^t(\FR(G))$ or, equivalently,  $\alpha(G,w)=\ls^{(t)}(G,w)$.
The following bounds are shown in \cite{LS91} for the $N$-index:
$${\sum_{i=1}^nw_i\over \alpha(G,w)}-2\le \rak_\LS (G,w)\le \DEF(G,w),\ \rak_\LS(G,w)\le |V(G)|-\alpha(G)-1,$$
where $\DEF(G,w)$ is as defined in (\ref{defect}).
Note the analogy with the bounds (\ref{lowerboundrk}), (\ref{upperboundrk})  and (\ref{2upper bound}) for the Handelman rank.
There is a shift of 2 between the two hierarchies which can be explained from the fact that the Lov\'asz-Schrijver construction starts from the fractional stable set polytope which already takes the edges into account, so that
$\ls^{(0)}(G,w)=\alpha^*(G,w)= \phan^{(2)}(G,w)$.
We also observe this shift by 2, e.g., in the results for perfect graphs and for odd cycles and wheels.
It seems moreover that the Handelman bound and the bound obtained by using the $N$-operator are closely related.
We did  some computational tests for the
graphs $K_4$, $W_5$ and $G_k$ ($k=2,3,4,5$) with different
weight functions; in all cases we observe that both bounds coincide, i.e.,
$\ls^{(1)}(G,w)=\phan^{(3)}(G,w)$ holds. Understanding the exact link between the two hierarchies of Handelman and of Lov\'asz-Schrijver is an interesting open question.

\newcommand{\qhan}{q_{\text{han}}}

\subsection{De Klerk and Pasechnik LP hierarchy}

Given a graph $G=(V,E)$ with adjacency matrix $A$, de Klerk and Pasechnik \cite{KP02} formulate its stability number via  the following copositive program:
\begin{equation*}\label{copo}
\alpha(G)=\min\{\lambda:\lambda(I+A)-ee^T\in \CC_n\},
\end{equation*}
which is based on the Motzkin-Straus formulation:
\begin{equation}\label{MS}
{1\over \alpha(G)} =\min_{x\in \Delta} x^T(I+A)x,
\end{equation}
where $\Delta=\{x\in \oR^V_+:\sum_{i=1}^nx_i=1\}$ is the standard simplex.
As problem (\ref{MS}) is the problem of minimizing the quadratic polynomial $q(x)=x^T(I+A)x$ over the simplex $\Delta$,
 one can follow the approach sketched in Section \ref{sechan} and define, for any $t\ge 2$,  the corresponding (simplex based) Handelman bound
$$\qhan^{(t)}= \max \{\lambda: (q-\lambda \sigma^2)\sigma^{t-2}\in \oR_+[x]\},$$
where $\sigma=\sum_{i=1}^n x_i$.
(Recall Lemma \ref{lemsimplex}.) It turns out that it can be computed explicitly since it is directly related to
 the following bound introduced in \cite{KP02}:
$$\zeta^{(t)}(G)= \min\{\mu: (\mu q-\sigma^2)\sigma^t\in \oR_+[x]\}$$
for any $t\ge 0$.
Indeed it  follows from the definitions that
$$\zeta^{(t)}\qhan^{(t+2)}=1 \ \text{ for } t\ge 0.$$ De Klerk and Pasechnik \cite{KP02} show that
$$\zeta^{(0)}(G)\ge\zeta^{(1)}(G)\ge\cdots\ge\lfloor\zeta^{(t)}(G)\rfloor=\alpha(G)$$ for $t\ge \alpha(G)^2{-1}$.
Moreover, Pe\~na, Vera and Zuluaga \cite{PVZ12} give
the following  closed-form expression for the parameter $\zeta^{(t)}(G)$:
$$\zeta^{(t)}(G)=\frac{{t+2\choose 2}}{{u\choose 2}\alpha(G)+uv}, \ \ \text{ where } t+2=u\alpha (G)+v \text{ with } u,v\in \oN \text{ and }  v<\alpha(G).$$
From this we see that
 $\zeta^{(t)}(G)=\infty$ if $t\le \alpha(G)-2$
and $\zeta^{(t)}(G)=\alpha(G)+1$ if $t=\alpha(G)^2-2$. Moreover,  $\alpha(G)\le \zeta^{(t)}(G) <\alpha(G)+1$
for any  $t\ge \alpha(G)^2-1$, with a strict inequality $\alpha(G)< \zeta^{(t)}(G) $ if $G$ is not a complete graph.
Hence, in contrast to the LP bounds based on the Handelman, Sherali-Adams and Lov\'asz-Schrijver constructions (which are exact at order $n$),   the LP copositive-based bound is never exact (except for the complete graph),  one needs to round it in order to obtain the stability number.

From the above discussion it follows that the LP copositive rank  ${\rm{rk}}_{\text{KP}}(G)$, which we define as the  smallest
integer $t$ such that $\lfloor\zeta^{(t)}(G)\rfloor=\alpha(G)$, can be determined exactly: ${\rm{rk}}_{\rm KP} (G)=\alpha(G)^2-1$ for any graph $G$.
We now observe that it cannot be compared with the (hypercube based) Handelman rank $\rk(G)$. Indeed,
  for the complete graph $G=K_n$,  we have ${\rm{rk}}_{\text{KP}} (K_n)=0$ while  $\rk(K_n)=n$.
On the other hand,  the graph $K_{1,n}$ has ${\rm{rk}}_{\text{KP}}(K_{1,n})= n^2-1$ and
$\rk(K_{1,n})=2$.
As another example, for the graph $G_k$ from Example \ref{exLT}, ${\rm{rk}}_{\text{KP}}(G_k)=k^2-1$ while $\rk(G_k)\le 4$.
Hence the ranks of the two hierarchies are not comparable.
These  examples also show that the ranks of the Lov\'asz-Schrijver and of the LP copositive hierarchies are not comparable, since
${\rm{rk}}_{\rm LS}(K_n)=n-2$ and ${\rm{rk}}_{\rm LS}(K_{1,n})=0$.

\newcommand{\CUT}{\text{\rm CUT}}

\section{The Handelman hierarchy for the maximum cut problem}\label{conclude}
In this paper we have studied  how the (hypercube based) Handelman hierarchy applies to the maximum stable set problem.
A main motivation for studying this hierarchy is that, due to its simplicity, it is easier to analyze than other hierarchies. We proved several properties that seem to indicate that there is a close relationship to the hierarchy of Lov\'asz-Schrijver, whose exact nature still needs to be investigated.
Another interesting open question is whether the Handelman rank is upper bounded in terms  the tree-width of the graph.

We now conclude with some observations clarifying how the Handelman hierarchy applies to the maximum cut problem.
 Given a graph $G=(V,E)$ with edge weights $w\in \oR^E$, the max-cut problem asks to find a partition $(V_1,V_2)$ of the node set $V$ so that the total weight of the edges cut by the partition is maximized;
it is NP-hard, already in the unweighted case \cite{Karp72}. As observed in \cite{PH11} the formulation (\ref{popmaxcut}) extends to the weighted case:
$${\rm mc}(G,w)=\max_{x\in[0,1]^n}  \sum_{i\in V}d_ix_i-2\sum_{ij\in E}w_{ij}x_ix_j, $$ 
setting $d_i=\sum_{j\in V: ij\in E} w_{ij}.$
As the polynomial $ p(x)=\sum_{i\in V}d_ix_i-2\sum_{ij\in E}w_{ij}x_ix_j $ is square-free the Handelman bound of order $t$ can be formulated as
$$\min\{\lambda: \lambda-p\in H_t\}.$$
We show below that it can be equivalently reformulated in a more explicit way in terms of suitable valid inequalities for the cut polytope.
We need some definitions. The cut polytope
 $\text{CUT}_n$ is defined as the convex hull of the vectors $(v_iv_j)_{1\le i<j\le n}$ for
all  $v\in \{\pm 1\}^n$. So $\CUT_n$ is a polytope in the space $\oR^{n\choose 2}$ indexed by the edge set of the complete graph $K_n$.
Given an integer $t\ge 2$, among all the inequalities that are valid for $\CUT_n$, we consider only those that are supported by at most $t$ points of $[n]$ and we let $P^{(t)}_n$ denote the polytope in $\oR^{n\choose 2}$ defined by all these selected inequalities. Clearly,
$\CUT_n\subseteq P^{(t)}_n$. Moreover, for $n\ne 4$,  equality $\CUT_n=P^{(t)}_n$ holds if and only if $t=n$ (since $\CUT_n$ has some facet defining ineqaulities supported by $n$ points).
The case $n=4$ is an exception since $\CUT_4=P^{(3)}_4$.

\begin{proposition}\label{propmc}
Let $t\ge 2$ and, given an edge weighted graph $(G,w)$, consider the above mentioned polynomial $p=\sum_{i\in V} d_ix_i-2\sum_{ij\in E}w_{ij}x_ix_j$.
The following equality holds:
$$\min\{\lambda: \lambda-p\in H_t\}= \max_{y\in P^{(t)}_n} \sum_{ij\in E} w_{ij}(1-y_{ij})/2.$$
\end{proposition}

\newcommand{\bH}{\overline{H}}
\newcommand{\bI}{\overline  {\mathcal I}}

\begin{proof}
It is convenient to use $\pm 1$ valued  variables $z$ instead of the $0/1$ valued variables $x$.
So we set $z_i=1-2x_i$ for $i\in [n]$.
Then $p(x)=q(z)$, after defining the polynomial
$q(z)=\sum_{ij\in E}w_{ij}(1-z_iz_j)/2$.
Moreover define the $\pm 1$ analogue of the Handelman set $H_t$ from (\ref{eqsettiH}):
$$\bH_t=\{\sum_{T\subseteq [n]: |T|=t}\sum_{I\subseteq T} c_{I,T} (1-z)^I(1+z)^{T\setminus I}: c_{I,T}\ge 0\}.$$
Furthermore let $\bI$ denote the ideal in the polynomial ring $\oR[z]$ generated by $z_i^2-1$ for $i\in [n]$, and let $\bI_t$ denote its truncation at degree $t$.
One can easily verify that $\lambda-p\in H_t$ if and only if $\lambda -q\in \bH_t$ which, in turn, is equivalent to
$\lambda-q\in \bH_t+\bI_t$. Therefore we have
$$\min\{\lambda: \lambda-p\in H_t\}=\min\{\lambda: \lambda -q\in \bH_t+\bI_t\}.$$
Now we apply LP duality and obtain that the last program is equal to
$$\max_L\{L(q): L(1)=1,\ L(f)\ge 0\ \forall f\in \bH_t,\ L(f)=0\ \forall f\in \bI_t\},$$
where the maximum is taken over all linear functionals $L:\oR[z]_t\rightarrow \oR$.
Finally, we use the fact that this maximization program is equal to the maximum of
$\sum_{ij\in E}w_{ij} (1-y_{ij})/2$ taken over all $y\in P^{(t)}_n$, which is shown in \cite{Lau03} (top of page 20).
This concludes the proof.
\qed\end{proof}

For instance, for $t=2$, $P^{(2)}_n=[-1,1]^{n\choose 2}$ (since $-1\le y_{ij}\le 1$ are the only inequalities on two points valid for $\CUT_n$).
Hence, by Proposition \ref{propmc},  the Handelman bound of order 2 is equal to $\sum_{ij\in E}|w_{ij}|$, as shown in \cite{PH11} for the case $w\ge 0$.
For $t=3$, $P^{(3)}_n$ is defined by the triangle inequalities $y_{ij}+y_{ik}+y_{jk}\ge -1$ and $y_{ij}-y_{ik}-y_{jk}\ge -1$ for all $i,j,k\in [n]$. Therefore,
for an edge weighted graph $G$ where $G$ has no $K_5$ minor,
 we find that the Handelman bound of order 3 is exact and returns the value of the maximum cut (since the triangle
inequalities suffice to describe the cut polytope of $G$, after taking projections). In particular, the Handelman rank is at most 3 for a weighted odd circuit, which
 answers an open question of \cite{PH12} (which shows the result in the unweighted case).
As a final observation, we find that the rank of the Handelman hierarchy for the maximum cut problem in $K_n$ is equal to $n$ for any $n\ne 4$ (which was shown in \cite{PH11} for $n$ odd).

\begin{acknowledgements}
We thank E. de Klerk and J.C. Vera for useful discussions. We also thank  two anonymous referees for their comments which helped  improve the clarity of the paper and for drawing our attention to the paper by Krivine \cite{Kr64}.
\end{acknowledgements}

\end{document}